\documentclass[11pt, a4paper]{amsart}
\usepackage{amssymb}
\usepackage{amsmath}
\usepackage{youngtab}
\usepackage{young}
\usepackage{xfrac}
\usepackage{ytableau}
\usepackage{young}
\usepackage{youngtab}
\usepackage{lmodern}
\usepackage{enumitem}
\usepackage{graphicx} 
\usepackage[T1]{fontenc}
\usepackage{amsthm}
\usepackage{soul}
\usepackage[maxbibnames=99,
backend=biber,
style=alphabetic,
doi=false,
url=false,giveninits=true
]{biblatex}
\usepackage{fullpage}
 \usepackage{amsfonts,mathrsfs}
\usepackage{bbm,amsmath,amssymb,amsfonts,graphicx,color,subfig,mathtools, verbatim}
   \usepackage[bookmarksopen=false,breaklinks=true,%
      plainpages=false,%
      hyperindex=true,pdfstartview=FitH,colorlinks=true,%
      pdfpagelabels=true,colorlinks=true,linkcolor=blue,%
      citecolor=red,urlcolor=green,hypertexnames=false%
      ]%
   {hyperref}

\renewcommand{\vartheta}{\theta}

\newcommand{\K}{\mathbb{K}}
\newcommand{\N}{\mathbb{N}}

\newcommand{\Z}{\mathbb{Z}}
\newcommand{\R}{\mathbb{R}}
\newcommand{\C}{\mathbb{C}}
\newcommand{\x}{\mathbf{x}}
\renewcommand{\leq}{\leqslant}

\usepackage{faktor}

\newcommand{\Sym}{S}
\usepackage{MnSymbol} 
\numberwithin{equation}{section}
\newtheorem{theorem}{Theorem}[section]
\newtheorem{lemma}[theorem]{Lemma}
\newtheorem{remark}[theorem]{Remark}

\newtheorem{corollary}[theorem]{Corollary}
\newtheorem{proposition}[theorem]{Proposition}
\newtheorem{definition}[theorem]{Definition}
\newtheorem{example}[theorem]{Example}
\newtheorem{conjecture}[theorem]{Conjecture}
\DeclareMathOperator{\sh}{sh}

\DeclareMathOperator{\Mon}{Mon}
\DeclareMathOperator{\Stab}{Stab}
\DeclareMathOperator{\len}{len}
\DeclareMathOperator{\spe}{sp}

\DeclareMathOperator{\wt}{wt}

\DeclareMathOperator{\BP}{BP}

\DeclareMathOperator{\cut}{cut}

\DeclareFontFamily{U}{mathx}{\hyphenchar\font45}
\DeclareFontShape{U}{mathx}{m}{n}{
      <5> <6> <7> <8> <9> <10>
      <10.95> <12> <14.4> <17.28> <20.74> <24.88>
      mathx10
      }{}
\DeclareSymbolFont{mathx}{U}{mathx}{m}{n}
\DeclareFontSubstitution{U}{mathx}{m}{n}
\DeclareMathAccent{\widecheck}{0}{mathx}{"71}

\ytableausetup{centertableaux}

\begin{document}

\title[]
{
The poset of Specht ideals for  hyperoctahedral groups
}

\author{Sebastian Debus}
\address{Fakultät für Mathematik, Otto-von-Guericke-Universität Magdeburg, 39106 Magdeburg, Germany}
\email{sebastian.debus@ovgu.de}

\author{Philippe Moustrou}
\address{
Institut de Mathématiques de Toulouse, UMR 5219, UT2J, 31058
Toulouse, France}
\email{philippe.moustrou@math.univ-toulouse.fr}

\author{Cordian Riener}
\address{Department of Mathematics and Statistics, UiT - the Arctic University of Norway, 9037 Troms\o, Norway}
\email{cordian.riener@uit.no}

\author{Hugues Verdure}
\address{
Department of Mathematics and Statistics, UiT - the Arctic University of Norway, 9037 Troms\o, Norway}
\email{hugues.verdure@uit.no}

\thanks{This work has been supported by European Union's Horizon 2020 research and innovation programme under the Marie Sk\l{}odowska-Curie grant agreement 813211 (POEMA), the Troms\o~ Research foundation grant agreement 17matteCR and by the project Pure Mathematics in Norway, funded by Trond Mohn Foundation and Tromsø Research Foundation}

\begin{abstract}
Specht polynomials classically realize the irreducible representations of the symmetric group. The ideals defined by these polynomials provide a strong connection with the combinatorics of Young tableaux and have been intensively  studied by several authors. We initiate similar investigations for the ideals defined by the Specht polynomials associated to the hyperoctahedral group $B_n$. We introduce a bidominance order on bipartitions which describes the poset of inclusions of these ideals and study algebraic consequences on general $B_n$-invariant ideals and varieties, which can lead to computational simplifications.    

\end{abstract}
\maketitle
\markboth{S.~Debus, P.~Moustrou, C.~Riener, and H.~Verdure}{Specht ideals for the hyperoctahedral group}

\section{Introduction}

Symmetries provide beautiful connections between algebra, geometry and efficient computations: on the one hand, the symmetries of geometrical objects can be described with the algebraic language of group theory, while on the other hand algebraic problems affording additional structure can be solved more efficiently once symmetry is appropriately taken into consideration.  
A particular incarnation of these phenomena occurs when studying algebraic systems of polynomial equations whose solution set is invariant under a group action. In this set-up, when looking at the corresponding polynomial ideal, the machinery of invariant and representation theory can be employed to gain information about the solutions of the initial system, and to simplify its resolution.

\medskip

These kinds of questions have been extensively studied in the literature for the \emph{symmetric group} $\mathcal{S}_n$, acting on the polynomial ring $\K[\mathbf{x}_1, \ldots, \mathbf{x}_n]$ over a field $\K$ by permuting variables. In particular, it has been observed in different computational tasks that the understanding of this action can lead to substantial algorithmic improvements (see for example \cite{computing1,computing2,computing3,computing4,computing5,computing6,computing7,computing8}). These improvements mostly build on the fact that in this situation, both representation and invariant theory are classically understood, and are closely related to the combinatorics of partitions and Young Tableaux.
More precisely, the irreducible representations of $\mathcal{S}_n$ are in bijection with the partitions of $n$, through a construction due to Specht: for every partition, one can define a polynomial whose $\mathcal{S}_n$-orbit spans an irreducible $\mathcal{S}_n$-module, called \emph{Specht module} \cite{specht1937darstellungstheorieSn}.
This motivates the study of \emph{Specht ideals}, the ideals generated by such modules, since they can be seen as building blocks of the action of the symmetric group on a polynomial ring.
The study of these objects has shown to be fruitful from various aspects, and the connection between these ideals and the combinatorics of partitions turns out to be deeper: not only there is a bijection between $\mathcal{S}_n$-Specht ideals and partitions of $n$, but this correspondence respects the poset structures. First results were proven by Li and Li \cite{li1981independence} and Haiman and Woo (see Woo's doctoral thesis \cite{woo2005ideals}), and then independently revisited and extended for algorithmic purposes in \cite{moustrou2021symmetric}.
In turn, this combinatorial understanding also provides information on these ideals from the point of view of commutative algebra: for instance they all are radical \cite{murai2021note}, and the partitions for which they are Cohen-Macaulay are understood \cite{yanagawa2021specht}.
The study of these ideals has also paved algorithmic ways to simplify calculations for $\mathcal{S}_n$-closed ideals and their corresponding varieties. They allow to understand the symmetry of the coordinates of the points in the variety, which in turn gives information on their dimension. This information can then be used to design more efficient algorithms by reducing the number of variables.

\medskip

In this article, we initiate a similar study for the  action of the \emph{hyperoctahedral group $B_n$} on a polynomial ring $\K[\mathbf{x}_1, \ldots, \mathbf{x}_n]$. The field $\K$ is assumed to be of characteristic $0$, although many results remain valid in positive characteristic. In this representation, this group can be seen as the group generated by permutations of variables and sign switches of variables, namely maps sending $\mathbf{x}_i$ to $-\mathbf{x}_i$. The group is isomorphic to the Weyl group of type $B$ and appears in several different areas, as hyperplane arrangements (\cite[Section 6.7]{aguiar2017topics},\cite{abe2020hessenberg}), representation theory \cite{chen1993derangements,chen1993induced,musili1993representations}, and has applications in the study of non negative even symmetric polynomials \cite{choi1987even,harris1999real} and optimization \cite{dostert2017new}.
Similarly to the case of permutations, this situation is profoundly connected to combinatorics. 
In this case, instead of partitions, the irreducible representations of $B_n$ are in bijection with \emph{bipartitions} of $n$. Furthermore, polynomial generators of the irreducible $B_n$-modules can be constructed in a similar way \cite{specht1937darstellungstheorieBn,musili1993representations}. We aim at a first investigation of the corresponding ideals  with the goal to extend the connections between algebra and combinatorics as far as possible. In contrast to the $\mathcal{S}_n$-case where there is a natural order on partitions, several orders are possible on bipartitions \cite{geissinger1978representations,al1981representations,dipper1995hecke}. However, while in the $\mathcal{S}_n$ case the poset of the standard order on partitions reflects the corresponding poset of ideal inclusions, none of the previously studied orders of bipartitions  satisfy this property. Therefore, we define another order on bipartitions. After studying the basic properties of this order, we are able to show that it indeed translates well to the ideal inclusion. Similarly to the case of $\mathcal{S}_n$, this combinatorial connection finds consequences for the corresponding varieties. In addition to the inclusion of varieties we are able to give a complete characterization in terms of orbit types of the points in these varieties. Further, this gives information on the possible orbit types of points in general $B_n$-invariant varieties, allowing for complexity reduction in the resolution of $B_n$-closed polynomial systems. 

\medskip

The paper is structured as follows: Section \ref{sec:Sn} overviews the situation of $\mathcal{S}_n$-Specht ideals. We initiate the study of $B_n$-Specht ideals in Section \ref{sec:Bn} with definitions and natural connections to the $\mathcal{S}_n$ case. In Section \ref{sec:posetbipartitions} we define our order for bipartitions and study its combinatorial properties. Following this, we show equivalence between our poset of bipartitions and the posets of Specht ideals and varieties in Section \ref{sec:inclusion}. In Section \ref{sec:varieties},  study possible decompositions of Specht varieties in terms of orbit types. Finally, we extend our study to general $B_n$-invariant ideals in Section \ref{sec:genideals}, before concluding the paper with closing remarks and open questions in Section \ref{sec:ciao}.

\section{\texorpdfstring{$\mathcal{S}_n$}{Sn}-Specht ideals}\label{sec:Sn}

\subsection{Definitions}
A \textit{partition} $\lambda=(\lambda_1,\ldots,\lambda_l)$ of $n$ is a sequence of non-increasing non negative integers such that $\sum_{i\geq1} \lambda_i = n$. 
We write $\lambda \vdash n$ when $\lambda$ is a partition of $n$ and say that $\emptyset$ is the unique partition of $0$. 
The \textit{size} of a partition $\lambda$ is $\mid\lambda\mid =\sum_{i\geq 1} \lambda_i$. 
The \textit{length} of a non-empty partition $\lambda=(\lambda_1,\lambda_2,\ldots) \vdash n$ is the maximal $l \in \N_0 := \N \uplus \{0\}$ with $\lambda_l > 0$, while the length of $\emptyset$ is $0$. We denote the length by $\len (\lambda)$. For any partition $\lambda=(\lambda_1,\ldots,\lambda_l) \vdash n$ we use the convention that $\lambda_s=0$ for every $s\geqslant l+1$. 

Let $\lambda,\mu \vdash n$ be partitions of the same size. Then $\lambda$ \textit{dominates} $\mu$ if and only if $\sum_{j=1}^k \lambda_j \geq \sum_{j=1}^k \mu_j$ for any $k$. We denote domination by $\mu \unlhd \lambda$.
A partition $\lambda$ can be represented via its \textit{(Young) diagram}, i.e., the ordered sequence of boxes from the left to the right and the top to the bottom, where the $i$-th line contains $\lambda_i$ many boxes. We say that the associated diagram has \textit{shape} $\lambda$. A \textit{tableau} of shape $\lambda$ is a filling of a diagram of shape $\lambda$ with all the numbers $[n] = \{1,\ldots,n\}$. Then, we write $\sh (T) = \lambda$ if $T$ is a tableau of shape $\lambda$.
For instance, $S=  \ytableausetup{smalltableaux}
\begin{ytableau}
9 & 3 & 6 & 4 \\
2 & 1 & 8\\ 
5 & 7
\end{ytableau}$ is a tableau of shape $(4,3,2)$. A \textit{generalized} tableau is a filling of a diagram with elements in $\K$. The \textit{conjugate partition} $\lambda^\perp$ of a partition $\lambda$ is the partition whose diagram is the one obtained from the diagram of $\lambda$ by interchanging the rows and columns. 

For a sequence $(i_1,\ldots,i_m)$ of natural numbers, we define the associated \textit{Vandermonde polynomial} in the variables $\mathbf{x}_{i_1},\ldots,\mathbf{x}_{i_m}$ as \[ \Delta_{(i_1,\ldots,i_m)}(\mathbf{x}) = \prod_{j < k \in  [m]}(\mathbf{x}_{i_j}-\mathbf{x}_{i_k}), \] while $\Delta_{(i)} = \prod_\emptyset (\mathbf{x}_{i_j}-\mathbf{x}_{i_k}) =1$.

\begin{definition}
Let $T$ be a tableau of shape $\lambda \vdash n$ with $m$ columns and let $T_i$ be the sequence of natural numbers containing the entries of the $i$-th column of $T$ from above to below. Then, the associated $\mathcal{S}_n$ \emph{Specht polynomial} $\spe_T(\mathbf{x})$ is the product of all the column Vandermonde polynomials of the columns, i.e., \[ \spe_{T}(\mathbf{x}) = \prod_{j=1}^m \Delta_{T_j}.\]
\end{definition}

For the tableau $S$ of shape $(4,3,2)$ above, we have 
\begin{align*}
\spe_{S}(\mathbf{x})&=\Delta_{(9,2,5)}(\x)\Delta_{(3,1,7)}(\x)\Delta_{(6,8)}(\x)\Delta_{(4)}(\x) 
\\ &=(\mathbf{x}_9-\mathbf{x}_2)(\mathbf{x}_9-\mathbf{x}_5)(\mathbf{x}_2-\mathbf{x}_5)(\mathbf{x}_3-\mathbf{x}_1)(\mathbf{x}_3-\mathbf{x}_7)(\mathbf{x}_1-\mathbf{x}_7)(\mathbf{x}_6-\mathbf{x}_8).
\end{align*}

\begin{definition} \label{def:Sn Specht}
Let $\lambda$ be a partition of $n$. We define the $\mathcal{S}_n$-\emph{Specht ideal} \[I_{\lambda} = \langle \spe_{T} (\mathbf{x}) : T \mbox{ is a tableau of shape } \lambda \rangle \subset \K[\mathbf{x}_1,\ldots,\mathbf{x}_n]\] and the $\mathcal{S}_n$-\emph{Specht variety} 
\[ V_{\lambda} = \{ a \in \K^n : f(a) = 0 \mbox{ for all } f \in I_{\lambda} \} \subset \K^n\]
 associated to $\lambda$.
 \end{definition}
 
 The group $\mathcal{S}_n$ acts transitively on the set of tableaux of shape $\lambda$, where an element $\sigma \in \mathcal{S}_n$ acts on a tableau $T$ by replacing every entry $i$ in a box by $\sigma (i)$. Thus, the $\mathcal{S}_n$-Specht ideal $I_{\lambda}$ is the ideal generated by the $\mathcal{S}_n$-orbit of a Specht polynomial of a tableau of shape $\lambda$.

\begin{definition}
For a partition $\lambda \vdash n$ we write $\mathcal{S}_\lambda = \mathcal{S}_{\lambda_1} \times \mathcal{S}_{\lambda_2} \times \cdots \subset \mathcal{S}_n$ and define the \emph{$\mathcal{S}_n$-orbit set} $H_\lambda = \{ z \in \K^n : \Stab_{\mathcal{S}_n}(z) \simeq \mathcal{S}_\lambda\}$. If $z \in H_\lambda$ we call $\lambda$ the \emph{$\mathcal{S}_n$-orbit type} of $z$.
\end{definition}

The orbit set of any partition is non-empty and the $H_\lambda$'s define a set partition of $\K^n$. For instance, $H_{(3,2,2,1)}$ is the $\mathcal{S}_n$ orbit of the set \[\{(a_1,a_1,a_1,a_2,a_2,a_3,a_3,a_4) \in \K^n  : a_i \neq a_j, ~\forall i \neq j\}.\]

\subsection{Inclusions and applications}

The dominance order for integer partitions is well studied and understood. We recall that if $(P,\preccurlyeq)$ is a poset and $p,q\in P$ then $p$ \textit{covers} $q$ if and only if $p \neq q$, $ q \preccurlyeq p$, and for any $r \in P$, $q \preccurlyeq r \preccurlyeq p$ implies $r \in \{p,q\}$.
Brylawski studied the lattice of integer partitions of $n$ with respect to the dominance order and classified the covering relations (\cite[Proposition 2.3]{brylawski1973lattice}). Let $\lambda,\mu \vdash n$ be partitions. 
Then, $\mu \unlhd \lambda$ is a covering if and only if $\lambda$ is of the form 
\begin{align*}
    \lambda= (\mu_1,\ldots,\mu_{i-1},\mu_{i}+1,\mu_{i+1},\ldots,\mu_{j-1},\mu_{j}-1,\mu_{j+1},\ldots,\mu_l),
\end{align*}
and either $j=i+1$ or $\mu_i = \mu_{j-1}$ 
(and $\mu_{i-1} > \mu_i$ and $\mu_{j} > \mu_{j+1}$ to ensure that $\mu$ is a partition). In particular, the diagram of shape $\mu$ can be obtained from the diagram of shape $\lambda$ via moving one box from the end of row $i$ to row $j$.
\begin{example}
The following are two coverings of partitions displayed by their diagrams. 
 \begin{center}
 \ytableausetup{smalltableaux}
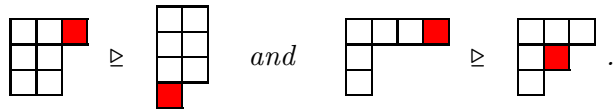
 \begin{ytableau}
*(white) & *(white) & *(red) \\
*(white) & *(white) \\
*(white) & *(white) 
\end{ytableau} 
$~~\unrhd ~~$
 \ytableausetup{smalltableaux}
 \begin{ytableau}
*(white) & *(white) \\
*(white) & *(white) \\
*(white) & *(white) \\
 *(red)
\end{ytableau} \quad and \quad \ytableausetup{smalltableaux}
 \begin{ytableau}
*(white) & *(white) & *(white) & *(red) \\
*(white)   \\
*(white) 
\end{ytableau} 
$~~\unrhd ~~$
 \ytableausetup{smalltableaux}
 \begin{ytableau}
*(white) & *(white) & *(white) \\
*(white) & *(red) \\
*(white)  
\end{ytableau}~.
\end{center}
\end{example}
The following theorem shows the equivalences of the posets of partitions with respect to dominance order, and the posets of $\mathcal{S}_n$-Specht ideals and varieties with respect to inclusion.
\begin{theorem}[\cite{moustrou2021symmetric},~Theorem~1]\label{thm:equivalenceSn}
Let $\lambda$ and $\mu$ be partitions of $n$. Let $I_{\lambda}, I_{\mu}$ denote their associated $\mathcal{S}_n$-Specht ideals and $V_{\lambda}, V_{\mu}$ their associated $\mathcal{S}_n$-Specht varieties. Then, the following assertions are equivalent:
\begin{enumerate}
    \item\label{equivalence_aSn} The partition $\lambda$ dominates $\mu$, i.e. $\lambda \unrhd \mu$;
    \item\label{equivalence_bSn} The $\mathcal{S}_n$-Specht ideal $I_{\lambda}$ contains the $\mathcal{S}_n$-Specht ideal $I_{\mu}$, i.e. $I_{\lambda} \supset I_{\mu}$;
    \item\label{equivalence_cSn} The $\mathcal{S}_n$-Specht variety $V_{\lambda}$ is contained in the $\mathcal{S}_n$-Specht variety $V_{\mu}$, i.e. $V_{\lambda} \subset V_{\mu}$.
    \end{enumerate}
\end{theorem}
The $\mathcal{S}_n$-Specht varieties can be decomposed using $\mathcal{S}_n$ orbit sets.
\begin{theorem}[\cite{moustrou2021symmetric},~Corollary~1] \label{thm:Sn orbit decomposition}
Let $\mu \vdash n$ be a partition. Then, the associated $\mathcal{S}_n$-Specht variety is
\[ V_\mu = \left( \bigcup_{\lambda \unlhd \mu} H_\lambda \right)^c = \bigcup_{\lambda \not \unlhd \mu}H_\lambda.\]
\end{theorem}
This characterization already shows that in general $\K[\mathbf{x}_1,\ldots,\mathbf{x}_n]/I_\lambda$ is not Cohen-Macaulay for a $\mathcal{S}_n$-Specht ideal $I_\lambda$, since the varieties are not equidimensional. Yanagawa classified the few cases when a Specht ideal is Cohen-Macaulay.

\begin{theorem}[\cite{yanagawa2021specht},~Corollary~4.4]
The ring $\K[\mathbf{x}_1,\ldots,\mathbf{x}_n] /I_\lambda$ is Cohen-Macaulay if and only if $\lambda$ is one of the following form
\begin{enumerate}
    \item $\lambda = (n-d,1,\ldots,1)$;
    \item $\lambda = (n-d,d)$;
    \item $\lambda = (a,a,1)$.
\end{enumerate}
\end{theorem}

The authors in \cite{woo2005ideals,murai2021note} prove that a $\mathcal{S}_n$-Specht ideal is radical. Their proof uses Theorem \ref{thm:Sn orbit decomposition}, i.e., that $\mathcal{S}_n$-Specht varieties can be written as disjoint unions of $\mathcal{S}_n$ orbit sets, and the non-emptyness of any orbit set $H_\lambda.$

\section{Definition and first properties of \texorpdfstring{$B_n$}{Bn}-Specht ideals}\label{sec:Bn}

The \textit{hyperoctahedral group} $B_n$ is the symmetry group of the $n$-dimensional hypercube. It can be written as the wreath product $ \mathcal{S}_2 \wr \mathcal{S}_n \simeq \{ \pm 1\}^n \ltimes \mathcal{S}_n$,  acting on the polynomial ring $\mathbb{K}[\mathbf{x}_1,\ldots,\mathbf{x}_n]$ in the following natural way. An element $(\tau,\rho) \in B_n$, where $\tau \in \{ \pm 1\}^n \cong \mathcal{S}_2^n$ and $\rho \in \mathcal{S}_n$, acts on a monomial $\mathbf{x}_i$ by $(\tau,\rho)\cdot \mathbf{x}_i = \tau_{\rho(i)} \cdot \mathbf{x}_{\rho (i)}$. 
We call an ideal $I \subset \mathbb{K}[\mathbf{x}]$ \textit{$B_n$-invariant} if for all $f \in I$ and all $\sigma \in B_n$ we have $\sigma \cdot f \in I$.

\subsection{Definitions}
 A \textit{bipartition} of $n$ is a pair $(\lambda,\mu)$, where $\lambda \vdash n_1$, $\mu \vdash n_2$ are partitions and $n_1 + n_2 = n$. We denote the set of all bipartitions of $n$ by $\BP_n$. A \textit{(Young) bidiagram} of a bipartition $(\lambda,\mu)$ is the pair of diagrams of shape $\lambda$ and $\mu$. A \textit{bitableau} is a filling of a bidiagram with all the numbers in $[n]$. We write $\sh (T,S) = (\lambda,\mu)$ if $(T,S)$ is a bitableau of shape $(\lambda,\mu)$.
For example, $(T',S')= \left( \ytableausetup{smalltableaux}
\begin{ytableau}
4 & 3 \\
2 \\ 
5
\end{ytableau} ~~~, ~~~
\ytableausetup{smalltableaux}
\begin{ytableau}
 6 \\ 1 \\ \none
\end{ytableau} \right)$ is a bitableau of shape $((2,1,1),(1,1))$.  When we consider representatives of $B_n$-orbits of points, we do not need to distinguish between the signs of coordinates. Thus, we write $\mathbf{x}^2 = (\mathbf{x}_1^2,\ldots,\mathbf{x}_n^2)$ and analogously $z^2=(z_1^2,\ldots,z_n^2)$ for points $z \in \K^n$. A \textit{generalized} bitableau is a filling of a bidiagram with elements in $\K$.

\begin{definition}
Let $(T,S)$ be a bitableau and let $T_i,\mathcal{S}_i$ be the sequences of natural numbers containing the entries of the $i$-th column of $T$ and $S$ from above to below. Then, the associated $B_n$ \emph{Specht polynomial} is 
\[\spe_{(T,S)}(\mathbf{x}) \quad = \quad \spe_{T}(\mathbf{x}^2) \spe_{S} (\mathbf{x}^2) \prod_{k \in S} \mathbf{x}_k \quad = \quad
\prod_{i\geq 1} \Delta_{T_i}(\mathbf{x}^2)\prod_{j \geq 1}  \Delta_{\mathcal{S}_j}(\mathbf{x}^2)  \prod_{k \in S} \mathbf{x}_k   \]
\end{definition}
\noindent where the notation $\spe_T$ is naturally adapted in this context to a Tableau $T$ which is not necessarily filled with the integers $1,\ldots,k$.
For the bitableau $(T',S')$ of shape $((2,1,1),(1,1))$ above, we have \[ \spe_{(T',S')}(\mathbf{x})=\Delta_{(4,2,5)}(\mathbf{x}^2)\cdot \Delta_{(3)}(\x^2)\cdot \Delta_{(6,1)}(\mathbf{x}^2)\cdot \mathbf{x}_6 \mathbf{x}_1 =  (\mathbf{x}^2_4-\mathbf{x}^2_2)(\mathbf{x}^2_4-\mathbf{x}^2_5)(\mathbf{x}^2_2-\mathbf{x}^2_5)(\mathbf{x}^2_6-\mathbf{x}^2_1)\mathbf{x}_1\mathbf{x}_6.\] 
The $B_n$ Specht polynomials are defined  in works of Specht \cite{specht1937darstellungstheorieBn} who proved that the vector space of all $B_n$ Specht polynomials associated with a fixed bipartition is an irreducible $B_n$-representation if the characteristic is $0$. Moreover, for pairwise different bipartitions the associated irreducible representations are non-isomorphic and all irreducible representations arise as the vector spaces of certain $B_n$ Specht polynomials. A constructive proof can also be found in~\cite[Chapter 7]{musili1993representations}.

From now on, if not specified, Specht polynomials will stand for $B_n$-Specht polynomials.

\begin{definition} \label{def:Bn Specht}
Let $(\lambda,\mu)$ be a bipartition of $n$. We define the $B_n$-\emph{Specht ideal} \[I_{(\lambda,\mu)} = \langle \spe_{(T,S)} (\mathbf{x}) : (T,S) \mbox{ is a bitableau of shape } (\lambda,\mu) \rangle \subset \K[\mathbf{x}_1,\ldots,\mathbf{x}_n]\] and the $B_n$-\emph{Specht variety} 
\[ V_{(\lambda,\mu)} = \{ z \in \K^n : f(z) = 0 \mbox{ for all } f \in I_{(\lambda,\mu)} \} \subset \K^n\]
 associated to $(\lambda,\mu)$. 
 \end{definition}

Again, the $B_n$-Specht ideal $I_{(\lambda,\mu)}$ is the ideal generated by the $\mathcal{S}_n$ orbit of a Specht polynomial of a bitableau of shape $(\lambda,\mu)$. We observe that switching of signs of variables in a Specht polynomial $\spe_{(T,S)}$ returns $\pm \spe_{(T,S)}$. Thus, $I_{(\lambda,\mu)}$ also equals the $B_n$ orbit of $\spe_{(T,S)}$.

\subsection{Relations between \texorpdfstring{$B_n$}{Bn} and \texorpdfstring{$\mathcal{S}_n$}{Sn} Specht polynomials}
\begin{definition}
Let $\lambda \vdash n_1,\mu \vdash n_2$ be partitions. Then, the \emph{glueing} of $\lambda$ and $\mu$ is the partition $\lambda \uplus \mu = (\lambda_1+\mu_1,\lambda_2+\mu_2,\ldots) \vdash n_1 + n_2$. The \emph{concatenation} $\lambda \vee \mu \vdash n_1+n_2$ is the partition obtained by rearranging $(\lambda_1,\ldots,\lambda_s,\mu_1,\ldots,\mu_t)$ in decreasing order.
\end{definition}
The glueing $\lambda \uplus \mu$ defines indeed again a partition. Since $\lambda_i \geq \lambda_{i+1}$ and $\mu_i \geq \mu_{i+1}$ we have $\lambda_i+\mu_i \geq \lambda_{i+1} + \mu_{i+1}$ for any $i$.

\begin{example}
The glueing of the partitions $ (3,2,2),(4,1)$ with diagrams 
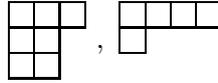
\[
\begin{ytableau}
*(white) & *(white) & *(white)\\
*(white)  & *(white) \\ 
*(white) & *(white)
\end{ytableau} ~~~, ~~~
\ytableausetup{smalltableaux}
\begin{ytableau}
 *(white) & *(white)  & *(white) & *(white) \\ *(white) \\ \none
\end{ytableau}~\] is the partition $(7,3,2)$ with diagram 
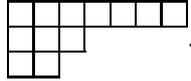
\[\begin{ytableau}
*(white) & *(white) & *(white) & *(white) & *(white) & *(white) & *(white) \\
*(white) & *(white) &*(white)  \\ 
*(white) & *(white)
\end{ytableau}.\] 
\end{example}

As observed in~\cite{hall1959algebra}, if $\lambda,\mu$ are two partitions, then $(\lambda \uplus \mu)^\perp = \lambda^\perp \vee \mu^\perp$. This observation provides a natural connection between bitableaux of shape $(\lambda, \mu)$ and tableaux of shape $\lambda \uplus \mu$.
Concretely, let $(T,S)$ be a bitableau of shape $(\lambda,\mu)$. Then, we can consider the tableau $T \uplus S$ of shape $\lambda \uplus \mu$, where the columns of $T \uplus S$ are filled like the columns of $T$ and $S$. When two columns in $\lambda \uplus \mu$ have the same length, they are ordered by their occurrence in the bitableau $(T,S)$ from the left to the right. For instance, for $(T^*,S^*)=  \left( \ytableausetup{smalltableaux}
\begin{ytableau}
1 & 2 & 10 & 9 \\
4 & 8 & 7 \\ 
6
\end{ytableau} ~~~, ~~~
\ytableausetup{smalltableaux}
\begin{ytableau}
 3 \\ 5 \\ \none
\end{ytableau}\right)$ we have $T^* \uplus S^* = \ytableausetup{smalltableaux}
\begin{ytableau}
1 & 2 & 10 & 3 & 9 \\
4 & 8 & 7 & 5  \\ 
6
\end{ytableau}~.$ Since this map is invertible we get:

\begin{proposition} \label{prop:gluing of tableau}
The tableaux of shape $\lambda \uplus \mu$ are in 1:1 correspondence with the bitableaux of shape $(\lambda,\mu)$. A bijection is given by $(T,S) \mapsto T \uplus S$. 
\end{proposition}

 The lemma below describes the connection between $\mathcal{S}_n$ and $B_n$-Specht polynomials. In particular, it motivates the definition of the same operations on the bidiagram of shape $(\lambda,\mu)$ and on the diagram of its glueing $\lambda \uplus \mu$ via moving some of the boxes in a diagram in Section \ref{sec:posetbipartitions}.

\begin{lemma}\label{lemma:spechtofsums}
Let $(\lambda,\mu) \in \BP_n$ be a bipartition and let $(T,S)$ be a bitableau of shape $(\lambda,\mu)$. Then \begin{eqnarray*}\spe_{(T,S)}(\mathbf{x}_1,\ldots \mathbf{x}_n) =   \spe_{T \uplus S}(\mathbf{x}_1^2,\ldots,\mathbf{x}_{n}^2)\prod_{j \in S} \mathbf{x}_j .\end{eqnarray*}
\end{lemma}
\begin{proof}
It is an immediate consequence of Proposition~\ref{prop:gluing of tableau}, since the Specht polynomials are defined as product of Vandermonde polynomials on the columns of the glued partition.
\end{proof}

\subsection{Existing orders on bipartitions}
As mentioned in the introduction, partial orders on the set of bipartitions of $n$ have been studied by several authors.
Let $(\lambda,\mu),(\lambda',\mu') \in \BP_n$ be bipartitions. 
For instance the following statements define partial orders: 
 \begin{equation} \label{order Hecke}
        (\lambda',\mu') \preccurlyeq (\lambda,\mu) \Leftrightarrow \begin{cases} \sum_{j \leq k} \lambda_j' \leq \sum_{j \leq k}\lambda_j , & \mbox{ for all } k, \mbox{ and } \\
    \mid\lambda'\mid + \sum_{j \leq k} \mu_j' \leq  \mid\lambda \mid + \sum_{j \leq k} \mu_j, & \mbox{ for all } k 
    \end{cases}
    \end{equation}
    was introduced in \cite{dipper1995hecke} to study Hecke algebras of type $B_n$, and was recently proven to occur naturally in the field of spin group theory \cite{xia2018partial}. Ariki generalized their order to multipartitions to study Hecke algebras of type $ G (m, 1, n) $ \cite{ariki2001classification}. The partial order
       \begin{equation} \label{order induced}
        (\lambda',\mu') \preccurlyeq (\lambda,\mu) \Leftrightarrow          \begin{cases} \mid\lambda'\mid < \mid\lambda\mid, & \mbox{ or } \\
             \mid\lambda'\mid = \mid\lambda\mid, & \mbox{ and } \lambda' \unlhd \lambda, \mu' \unlhd \mu 
             \end{cases}
    \end{equation}
    was formalized in \cite{al1981representations} to construct $B_n$-irreducible representations based on a more general procedure valid for finite groups. In~\cite{geissinger1978representations}, the authors introduce a partial order on bipartitions that extends the dominance order on partitions in the following way: \begin{equation}\label{order Geissinger}(\lambda',\mu') \preccurlyeq (\lambda,\mu) \Leftrightarrow \chi_{(\lambda,\mu)}-\chi_{(\lambda',\mu')} \textrm{ is zero or proper},\end{equation} where $\chi_{(\lambda,\mu)}$ is the character of $\operatorname{Ind}_{\mathcal{S}(\lambda) \times B(\mu)}^{B(n)} (1)$. The partial orders~(\ref{order Hecke}),~(\ref{order induced}) and~(\ref{order Geissinger}) are not equivalent.
    
    Moreover, these orders do not capture inclusions of ideals and varieties. Namely, for~(\ref{order Hecke}) and~(\ref{order induced}), we have the following ordering of bipartitions of $n=2$: \[ ((2),\emptyset) \succ ((1,1),\emptyset) \succ ((1),(1)),\] and for~(\ref{order Geissinger}), \[((2),\emptyset) \succ ((1,1),\emptyset) \textrm{ and }((1),(1)) \succ ((1,1),\emptyset)\]while the corresponding ideals are \[I_{((1,1),\emptyset)} = <\mathbf{x}_1^2-\mathbf{x}_2^2> \subsetneq I_{((1),(1))} = <\mathbf{x}_1,\mathbf{x}_2> \subsetneq I_{((2),\emptyset)} = <1>.\] In the next section, we introduce a new order on bipartitions that will capture inclusion of Specht ideals.

\section{The poset of bipartitions}\label{sec:posetbipartitions}

In this section we introduce our new order for bipartitions:
\begin{definition}
Let $(\lambda,\mu),(\lambda',\mu') \in \BP_n$ be biparitions of $n$. We say that $ (\lambda,\mu)$ \emph{bidominates} $(\lambda',\mu')$ if and only if \begin{align*}
     \sum_{j=1}^k (\lambda_j' + \mu_j') \quad & \leq \quad \sum_{j=1}^k (\lambda_j + \mu_j), \quad \mbox{and}\\ 
    \sum_{j=1}^{k-1} (\lambda_j' + \mu_j') + \lambda_{k}' \quad & \leq \quad \sum_{j=1}^{k-1} (\lambda_j + \mu_j )+ \lambda_{k} 
\end{align*} for all positive integers $k$. If $(\lambda,\mu)$ bidominates $(\lambda',\mu')$ we write $(\lambda',\mu') \unlhd (\lambda,\mu)$. We call $\unlhd$ the \emph{bidominance order}.
\end{definition}

We point out that the first condition is just a condition on the glueing of the bipartitions, i.e., \[\lambda' \uplus \mu' \unlhd \lambda \uplus \mu.\]

\begin{example}
The following bipartitions of $8$ are comparable: $ ((2,1,1),(3,1)) \unlhd ((3,2),(2,1)) $, since \begin{align*}
    2 \leq 3, \quad 5 \leq 5, \quad 6 \leq 7, \quad
    7 \leq 8, \quad 8 \leq 8. 
\end{align*}
However, the bipartitions $((2),(1,1))$ and $(\emptyset,(4))$ are not comparable, since $2 > 0 $ but $3 < 4$.
\end{example}

Although we use the same symbol for dominance and bidominance, this should not create any confusion, as they are defined on sets with empty intersection. We identify bipartitions with their associated bidiagrams and speak about \textit{boxes} in a bipartition.\\
It follows from the definition that our bidominance order is a partial order on $\BP_n$, and the previous example shows that it is not a total order.

Before proving our main theorem in the next section, we need a better understanding of our poset of bipartitions. 

The smallest element in $(\BP_n, \unlhd)$ is $(\emptyset,(1,\ldots,1))$, while the largest element is $((n),\emptyset)$. 
The following theorem characterizes the covering relations in the poset $(\BP_n,\unlhd)$. It turns out that there are four different cases, that are illustrated in Example \ref{example:covering cases}.

\begin{theorem} \label{thm:Bn covering classification}
Let $(\lambda,\mu),(\lambda',\mu') \in \BP_n$ be bipartitions and let $i = \min \{ j \in [n] : (\lambda_j,\mu_j) \neq (\lambda'_j,\mu'_j)\}$. Then, $(\lambda,\mu)$ covers $(\lambda',\mu')$ if and only if one of the following statements is true:
\begin{enumerate}
    \item $\mu=\mu'$, $\lambda$ covers $\lambda'$ with respect to the dominance order on partitions with  $\lambda_i' = \lambda_{i}-1$, and for $k$ such that $\lambda_k' = \lambda_k+1$, we have $\mu_{i-1} = \mu_i =\cdots = \mu_k$; 
    \item $\lambda=\lambda'$, $\mu$ covers $\mu'$ with respect to the dominance order on partitions, with $\mu_i' = \mu_i-1$, and for $k$ such that $\mu_k' = \mu_k+1$, we have $\lambda_i = \lambda_{i+1} = \cdots = \lambda_{k+1}$; 
 
    \item $\lambda\neq \lambda'$, $\mu\neq \mu'$ and $\lambda_i > \lambda_i'$.  If $k$ is maximal with $\lambda_i = \lambda_k$, then $\mu_i = \mu_k$, $(\lambda'_j,\mu'_j) = (\lambda_j-1,\mu_j+1)$ for any integer $i \leq j \leq k$, and $(\lambda'_j,\mu'_j) = (\lambda_j,\mu_j)$ otherwise;

    \item $\lambda\neq \lambda'$, $\mu \neq \mu'$, $\lambda_i = \lambda'_i$ (and therefore $\mu_i>\mu'_i$). If $k$ is maximal with $\mu_i = \mu_k$, then $\lambda_{i+1} = \lambda_{k+1}$, $(\mu'_j,\lambda_{j+1}') = (\mu_j-1,\lambda_{j+1}+1)$ for any integer $i \leqslant j \leqslant k$ and there is equality otherwise. 
\end{enumerate}
\end{theorem}

The example below shows instances for all the covering cases of bipartitions. The boxes that are moved are colored in red.  
\begin{example} \label{example:covering cases}
\begin{itemize}
\item An example of a covering of type (1), where $i=2$ and $k=4$:
 \begin{center}

 $  \left(  \ytableausetup{smalltableaux}
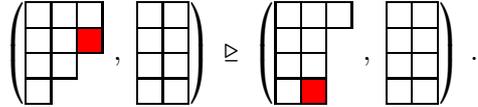
\begin{ytableau}
*(white) & *(white) & *(white)\\
*(white) & *(white) & *(red) \\
*(white) & *(white) \\
*(white)
\end{ytableau} ~~~, ~~~
\ytableausetup{smalltableaux}
\begin{ytableau}
 *(white) & *(white)  \\
*(white) & *(white)\\
*(white) & *(white)  \\
*(white) & *(white) 
\end{ytableau}  \right) ~~\unrhd ~~
\left( \ytableausetup{smalltableaux}
\begin{ytableau}
*(white) & *(white) & *(white)\\
*(white) & *(white)  \\
*(white) & *(white) \\
*(white) & *(red)
\end{ytableau} ~~~, ~~~
\ytableausetup{smalltableaux}
\begin{ytableau}
 *(white) & *(white)  \\
*(white) & *(white)\\
*(white) & *(white)  \\
*(white) & *(white) 
\end{ytableau}\right)~. $
\end{center}

\item An example of a covering of type (2), where $i=1$ and $k=4$:
 \begin{center}
 $ \left(\ytableausetup{smalltableaux}
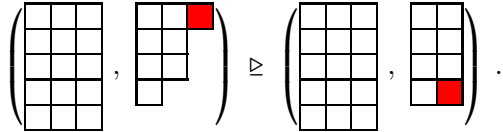
 \begin{ytableau}
*(white) & *(white) & *(white)  \\
*(white) & *(white) & *(white) \\
*(white) & *(white) & *(white) \\
*(white) & *(white) & *(white) \\
 *(white) & *(white) & *(white) 
\end{ytableau} ~~~, ~~~
\ytableausetup{smalltableaux}
\begin{ytableau}
 *(white) & *(white) & *(red) \\
*(white) & *(white)\\
*(white) & *(white)  \\
*(white)  \\
\none
\end{ytableau}  \right)
~~\unrhd ~~
\left( \ytableausetup{smalltableaux}
\begin{ytableau}
*(white) & *(white) & *(white) \\
*(white) & *(white) & *(white) \\
*(white) & *(white) & *(white) \\
*(white) & *(white) & *(white) \\
 *(white) & *(white) & *(white) 
\end{ytableau} ~~~, ~~~
\ytableausetup{smalltableaux}
\begin{ytableau}
 *(white) & *(white)  \\
*(white) & *(white)\\
*(white) & *(white)  \\
*(white)  & *(red) \\
\none
\end{ytableau} \right)~.$
\end{center}

\item An example of a covering of type (3), where $i=2$:
 \begin{center}
$\left( \ytableausetup{smalltableaux}
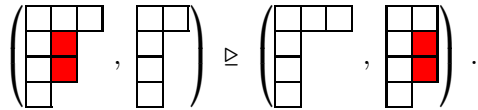
 \begin{ytableau}
*(white) & *(white) & *(white) \\
*(white) & *(red) \\
*(white) & *(red) \\
*(white) 
\end{ytableau} ~~~, ~~~
\ytableausetup{smalltableaux}
\begin{ytableau}
 *(white) & *(white) \\
*(white) \\
*(white)  \\
*(white)  
\end{ytableau}  \right)
~~\unrhd ~~
\left( \ytableausetup{smalltableaux}
 \begin{ytableau}
*(white) & *(white) & *(white) \\
*(white)  \\
*(white)  \\
*(white) 
\end{ytableau} ~~~, ~~~
\ytableausetup{smalltableaux}
\begin{ytableau}
 *(white) & *(white) \\
*(white) & *(red) \\
*(white) & *(red) \\
*(white)  
\end{ytableau} \right) ~.$
\end{center}

\item An example of a covering of type (4), where $i=1$:
 \begin{center}
$\left( \ytableausetup{smalltableaux}
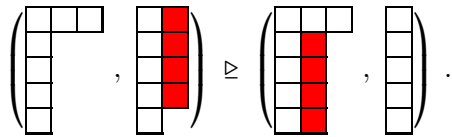
 \begin{ytableau}
*(white) & *(white) & *(white) \\
*(white)  \\
*(white)  \\
*(white) \\
*(white) 
\end{ytableau} ~~~, ~~~
\ytableausetup{smalltableaux}
\begin{ytableau}
 *(white) & *(red) \\
 *(white) & *(red)  \\
 *(white) & *(red)   \\
 *(white) & *(red)   \\
 *(white)
\end{ytableau}  \right)
~~\unrhd ~~
\left( \ytableausetup{smalltableaux}
 \begin{ytableau}
*(white) & *(white) & *(white) \\
*(white)  & *(red)  \\
*(white)  & *(red) \\
*(white) & *(red) \\
*(white)  & *(red)
\end{ytableau} ~~~, ~~~
\ytableausetup{smalltableaux}
\begin{ytableau}
 *(white) \\
*(white) \\
*(white) \\
*(white)  \\
*(white) 
\end{ytableau} \right) ~.$
\end{center} 
\end{itemize}
\end{example}

One can think of the cases (3) and (4) as moving a \textit{partial column} from the left diagram in the bidiagram $(\lambda,\mu)$ to the right side and staying in the same row, or moving a partial column from the right diagram to the left and going down one row.

Now we present a proof of the theorem.

\begin{proof}[Proof of Theorem \ref{thm:Bn covering classification}]
We start the proof by showing that operations $(1)-(4)$ define covering relations. Suppose that $(\lambda',\mu')$ is obtained from $(\lambda,\mu)$ by one of these operations. We need to show that $(\lambda,\mu)$ covers $(\lambda',\mu')$, that is, if $(\lambda^*,\mu^*)$ is such that $(\lambda,\mu) \unrhd (\lambda^*,\mu^*) \unrhd (\lambda',\mu')$, then either $(\lambda^*,\mu^*)=(\lambda,\mu)$ or $(\lambda^*,\mu^*)=(\lambda',\mu')$. Since the proofs for operations $(2)$ and $(4)$ are respectively similar to $(1)$ and $(3)$, we will focus on these two operations. \begin{enumerate}[leftmargin=*]
    \item[(A)]Suppose that $(\lambda',\mu')$ is obtained from $(\lambda,\mu)$ by operation $(1)$. In particular, $\mu=\mu'$, and there exists $1<i<k$ such that $\lambda'_i=\lambda_i-1$, $\lambda'_k=\lambda_k+1$, $\lambda'_j=\lambda_j$ for $j\neq i,k$ and $\mu_{i-1}=\cdots=\mu_k$. 
    It is not difficult to show, by taking the difference between two consecutive partial sums, that \[\forall j \in \{1,\ldots, i-1,k+1,\ldots\}   ,\ \lambda_j=\lambda^*_j=\lambda'_j \textrm{ and } \mu_j=\mu^*_j=\mu'_j,\] as well as $\mu^*_k = \mu_k=\mu'_k$. Since $\mu^*_{i-1} = \mu_{i-1}=\mu_k=\mu^*_k$, this implies that $\mu^*_{i-1}=\cdots=\mu^*_k$ as well. In turn, this means that \[\overline{\lambda}=(\lambda_i,\ldots,\lambda_k) \unrhd \overline{\lambda^*}=(\lambda^*_i,\ldots,\lambda^*_k) \unrhd (\lambda'_i,\ldots,\lambda'_k)=\overline{\lambda'}.\]  By hypothesis, $\overline{\lambda}$ covers $\overline{\lambda'}$, so that either $\overline{\lambda^*} = \overline{\lambda}$ or $\overline{\lambda^*}=\overline{\lambda'}$. In turn, this shows that $(\lambda,\mu)$ covers $(\lambda',\mu')$. 
    \item[(B)] Now, we assume that $(\lambda',\mu')$ is obtained from $(\lambda,\mu)$ by operation $(3)$. This means, that there exists $i \leqslant k$ such that $\lambda_i = \lambda_k > \lambda_{k+1}$, $\mu_{i-1}>\mu_i = \mu_k$ and $\lambda_j' = \lambda_j-1$, $\mu_j' = \mu_j+1$, for $i \leq j \leq k$ and otherwise $\lambda_j' = \lambda_j$ and $\mu_j' = \mu_j$. \\
    As above we observe that \[ \forall j < i \mbox{ or } j > k, \lambda_j = \lambda_j^* = \lambda_j' \mbox{ and } \mu_j = \mu_j^* = \mu_j'. \] In the same way, it is easy to show that \[\forall 
  j, \ i\leqslant j \leqslant k,\ \lambda_j+ \mu_j \geqslant \lambda^*_j + \mu^*_j \geqslant \lambda'_j + \mu_j' = \lambda_j+\mu_j,\] and \[\forall j, \ i\leqslant j \leqslant k,\ \lambda_j \geqslant \lambda^*_j \geqslant \lambda'_j=\lambda_j-1.\] Together, this implies that \[\forall j, \ i\leqslant j \leqslant k,\ \lambda^*_j \in \{\lambda_j,\lambda_j-1\} \textrm{ and } \mu^*_j \in \{\mu_j,\mu_j+1\}.\]
    \begin{enumerate}
        \item Assume first that $\lambda_i^* = \lambda_i -1$. Then for $i \leqslant j \leqslant k$, \[ \lambda_j-1  \leqslant \lambda^*_j \leqslant \lambda_i^* =\lambda_i-1=\lambda_j-1\] which implies  $\lambda^*_j=\lambda_j-1$ and $\mu^*_j=\mu_j+1$, that is, $(\lambda^*,\mu^*)=(\lambda',\mu')$.
        \item On the other hand, if $\lambda_i^* = \lambda_i$ or equivalently $\mu^*_i=\mu_i$, then for $i \leqslant j \leqslant k$, we have \[\mu_j \leqslant \mu^*_j \leqslant \mu^*_i = \mu_i=\mu_j\] which implies that $\mu_j^*=\mu_j$ and $\lambda^*_j=\lambda_j$, that is $(\lambda^*,\mu^*)=(\lambda,\mu)$.
    \end{enumerate}
\end{enumerate}
Thus, they describe a covering relation in the poset $(\BP_n,\unlhd)$. \smallskip

Now, we prove the converse. Let $(\lambda,\mu)$ and $(\lambda',\mu')$ be two different bipartitions of $n$ and assume that $(\lambda',\mu') \unlhd (\lambda,\mu)$. We show that there exists a bipartition $(\lambda^*,\mu^*)$ of $n$ that can be obtained from $(\lambda,\mu)$ through one of the cases (1)-(4), and $(\lambda',\mu') \unlhd (\lambda^*,\mu^*) \unlhd (\lambda,\mu)$. 
Let $i \in \N$ be minimal with $(\lambda_i,\mu_i) \neq (\lambda_i',\mu_i')$. \begin{enumerate}[leftmargin=*]
    \item[(A)] We  consider first the case where $\lambda_i > \lambda_i'$ and show that we can obtain $(\lambda^*,\mu^*)$ using one of the operations (1),(3), or (4). Let $k \in \N$ be maximal with $\lambda_i = \lambda_k$. 

\noindent We begin our analysis with distinguishing between the following two cases. Either there exists a $p \in \N$ such that $i \leq p \leq k$ and $\mu_p < \mu_{p-1} $ or not, with the convention that $\mu_1<\mu_0$. \smallskip

\begin{enumerate}[leftmargin=5pt]
    \item First, we assume that there exists such a $p$ and we fix the minimal $p$ with this property. Then $\mu_p$ is the first place after $\mu_{i-1}$ where we can put a box to still obtain a partition. Let $q \in \N$ be minimal such that $p \leq q \leq k$ and $\mu_q = \cdots = \mu_k$. We define $(\lambda^*,\mu^*)$ as the bipartition of $n$ with $(\lambda_j^*,\mu_j^*) = (\lambda_j-1,\mu_j+1)$ for every $q \leq j \leq k$ and otherwise $(\lambda_j^*,\mu^*_j) = (\lambda_j,\mu_j)$. We observe easily that $(\lambda^*,\mu^*) \unlhd (\lambda,\mu)$ and $\lambda^* \uplus \mu^* = \lambda \uplus \mu \unrhd \lambda' \uplus \mu'$, and we are just left with verifying \[\lambda_{t}^*+\sum_{j=1}^{t-1}(\lambda_j^*+\mu_j^*) \geq \lambda_{t}'+\sum_{j=1}^{t-1}(\lambda_j'+\mu_j'),\] for any $t \in \N$. However, this is clear for any $t < q$ and $t>k$. If $q \leqslant t \leqslant k$, we have \[\lambda_t' \leqslant \lambda_i' \leqslant \lambda_i-1 = \lambda_t-1 = \lambda_t^*\] and \[\sum_{j=1}^{t-1} (\lambda_j^*+\mu_j^*) = \sum_{j=1}^{t-1} (\lambda_j + \mu_j ) \geq \sum_{j=1}^{t-1} (\lambda_j'+\mu_j')\] so that  we also have \[\lambda_t^* +\sum_{j=1}^{t-1} (\lambda_j^* +\mu_j^*) \geq \lambda_t' + \sum_{j=1}^{t-1} (\lambda_j'+\mu_j').\] This is operation $(3)$.

\item Next, we assume that no such $p$ exists. In particular, $i>1$ and $\mu_{i-1} = \cdots = \mu_k$. We consider the closest possible free place in the bidiagram $(\lambda,\mu)$, namely we take $r > k$ to be the minimal integer with $\lambda_r < \lambda_k -1$ or ($\lambda_r = \lambda_k-1$ and $\mu_r < \mu_k$). Such an $r$ always exists, since we allow ourselves to extend the partitions with empty rows. If it did not exist, that would mean that $\lambda_k=1$ and $\mu_k=0$. By definition of $k$ and $i$, this would mean that $\lambda'_j=\lambda_j$ for $1 \leqslant j < i$, $\mu'_j=\mu_j$ for $1 \leqslant j <i$, that $1=\lambda_k= \cdots =\lambda_i >\lambda'_i = 0$. Also, we have $0=\mu_k = \dots=\mu_{i-1} = \mu_{i-1}'$. That means that the sizes of the bipartitions $(\lambda,\mu)$ and $(\lambda',\mu')$ are different, which is absurd. We proceed again with a case distinction. \\
\begin{enumerate}[leftmargin=7pt]
    \item Let us start with assuming that $\lambda_r = \lambda_k-1$ and $\mu_r < \mu_k$. We define $\mu_j^* = \mu_j-1$ and $\lambda_{j+1}^* = \lambda_{j+1}+1$ for all $k \leq j \leq r-1$, while $\mu_j^* = \mu_j$ and $\lambda_{j+1}^* = \lambda_{j+1}$ for any other $j \in \N_0$. Since $\lambda_{k+1}< \lambda_k$ and $\mu_{r-1} =\mu_k >\mu_r$, $(\lambda^*,\mu^*)$ is a bipartition. Clearly $(\lambda^*,\mu^*) \unlhd (\lambda,\mu)$. By construction we have \[\lambda_{t}^*+\sum_{j=1}^{t-1}(\lambda_j^*+\mu_j^*) = \lambda_{t}+\sum_{j=1}^{t-1}(\lambda_j+\mu_j) \geq \lambda_{t}'+\sum_{j=1}^{t-1}(\lambda_j'+\mu_j')\] for any $t \in \N$. Also, for $t<k$ or $t \geqslant r$, we have \[\sum_{j=1}^t (\lambda_j + \mu_j ) = \sum_{j=1}^t (\lambda_j^* + \mu_j^* )\] Thus, it remains to show this inequality holds for $k \leqslant t <r$. The following inequalities follow from the definitions of $k$, $i$ and $r$: \[\begin{array}{ll}\lambda_k^* = \lambda_k = \lambda_i \geqslant \lambda_i'+1 \geqslant \lambda_k' +1,\\  \lambda_t^* = \lambda_t + 1 = \lambda_k-1+1 \geqslant \lambda_k' +1 \geqslant \lambda_t'+1  & \textrm{ for } k<t\leqslant r \\ \mu_t^*+1 = \mu_t = \mu_k = \mu_{i-1} = \mu'_{i-1} \geqslant \mu_t' & \textrm{ for } k \leqslant t < r.\end{array} \] This shows that $\lambda^* \uplus \mu^* \unrhd \lambda' \uplus \mu'$, and we obtain $(\lambda^* ,\mu^*)$ from $(\lambda,\mu)$ by operation $(4)$.

\item Finally, we assume that $\lambda_r < \lambda_k-1$. In particular, $\mu_{i-1} = \cdots = \mu_{r-1}$ and $\lambda_{r-1} = \cdots = \lambda_{k+1} = \lambda_{k}-1 = \cdots = \lambda_{i}-1.$ We distinguish between two cases. 
\begin{enumerate}[leftmargin=9pt]
    \item[(ii,a)] First, assume that $\mu_{r-1} > \mu_r$. We define $\lambda^*_r = \lambda_r+1 \leqslant \lambda_{r-1}$ and $\mu_{r-1}^* = \mu_{r-1}-1 \geqslant \mu_r$, while $\mu_j^* = \mu_j$ and $\lambda_j^* = \lambda_j$ otherwise. Then $(\lambda^*,\mu^*)$ is a well-defined bipartition, and as usual, $(\lambda^*,\mu^*) \unlhd (\lambda,\mu)$. Also, proving that $(\lambda',\mu') \unlhd(\lambda^* ,\mu^*)$ is straightforward, except maybe proving that \[\sum_{j=1}^{r-1}(\lambda_j^* + \mu_j^*) \geqslant  \sum_{j=1}^{r-1} (\lambda'_j + \mu'_j).\] But as previously, we have: \[\begin{array}{ll}\lambda^*_j=\lambda_j=\lambda_j' & \textrm{ for } 1 \leqslant j  <i \\\mu^*_j=\mu_j=\mu_j' & \textrm{ for } 1 \leqslant j <i \\ \lambda_j^* = \lambda_j = \lambda_i \geqslant \lambda_i'+1 \geqslant \lambda'_j+1 &\textrm{ for } i\leqslant  j \leqslant k \\
\lambda_j^* = \lambda_j = \lambda_k-1 = \lambda_i-1 \geqslant \lambda_i' \geqslant \lambda'_j & \textrm{ for } k< j\leqslant r-1 \\ \mu_j^* = \mu_j = \mu_{i-1} = \mu'_{i-1} \geqslant \mu'_j & \textrm{ for } i \leqslant j \leqslant r-2 \\ \mu_{r-1}^* =  \mu_{i-1}-1 = \mu'_{i-1}-1 \geqslant \mu_{r-1}'-1 \end{array}\] All together, this gives \[\sum_{j=1}^{r-1} (\lambda^*_j + \mu^*_j) \geqslant \sum_{j=1}^{r-1} (\lambda'_j+ \mu_j') + (k-i) \geqslant \sum_{j=1}^{r-1} (\lambda'_j+ \mu_j')\] as wanted.  This also means that we obtain $(\lambda ^*,\mu^*)$ from $(\lambda,\mu)$ by operation $(4)$.
\item[ii,b] We are left with the case $\mu_r= \mu_{r-1}$. By a previous remark, this means that  $\mu_{i-1} = \cdots = \mu_r$. We set $\mu^* = \mu$ and $\lambda_k^* = \lambda_k-1,\lambda_r^* = \lambda_r+1$, and otherwise $\lambda_j^* = \lambda_j$. Then, by assumption $(\lambda^*,\mu^*)$ is a bipartition and $(\lambda^*,\mu^*) \unlhd (\lambda,\mu)$. It is also straightforward to show that $(\lambda^*,\mu^*) \unrhd (\lambda',\mu')$ except maybe the partial sums inequalities in rows $k$ to $r$. Since $\lambda^*_k + \lambda^*_r = \lambda_k+\lambda_r$, we only need to look at rows $k$ to $r-1$. To this purpose, we remark that: \[\begin{array}{ll}
\lambda_j^* =\lambda_j = \lambda'_j & \textrm{ for } 1 \leqslant j <i \\
\mu_j^* =\mu_j = \mu'_j & \textrm{ for  every }  j  \\
\lambda_j^* = \lambda_j = \lambda_i \geqslant \lambda'_i+1 \geqslant \lambda'_j+1 & \textrm{ for } i \leqslant j <k \\
\lambda^*_k = \lambda_k-1 = \lambda_i-1 \geqslant \lambda'_i \geqslant \lambda_k' \\
\lambda_j^* = \lambda_j = \lambda_k-1 = \lambda_i-1 \geqslant\lambda'_i \geqslant \lambda'_j & \textrm{ for } k<j\leqslant r-1
\end{array}\] Then for any $k \leqslant j <r$, we have \[\sum_{t=1}^j(\lambda^*_t + \mu^*_t)  \geqslant \sum_{t=1}^j(\lambda'_t + \mu'_t)  + (k-i) \geqslant \sum_{t=1}^j(\lambda'_t + \mu'_t) \]
and in the same way  \[\lambda^*_j + \sum_{t=1}^{j-1}(\lambda^*_t + \mu^*_t)   \geqslant \lambda'_j + \sum_{t=1}^{j-1}(\lambda'_t + \mu'_t). \] Thus we obtain $(\lambda^*,\mu^*)$ from $(\lambda,\mu)$ by operation $(1)$.
 \end{enumerate}
 \end{enumerate}
 \end{enumerate}
 \item[(B)] It remains to deal with the case $\lambda_i = \lambda'_i$, and $\mu_i> \mu'_i$. It can easily be deduced from the previous case, by noticing the following: let $\rho=(\mu_1,\mu_1,\mu_2,\ldots)$ and $\rho'=(\mu_1,\mu'_1,\mu'_2,\ldots)$, then $(\rho,\lambda)$ and $(\rho',\lambda')$ are bipartitions of $n+\mu_1$ such that $(\rho,\lambda)\unrhd (\rho',\lambda')$. If $j$ is minimal such that $(\rho_j,\lambda_j) \neq (\rho'_j,\lambda'_j)$, then $j=i+1$ and $\rho_j=\mu_i > \mu'_i = \rho'_j$. From what we have just seen, there exists $(\rho^*,\lambda^*)$, obtained from $(\rho,\lambda)$ by operations $(1)$, $(3)$ or  $(4)$, such that $(\rho,\lambda) \unrhd (\rho^*,\lambda^*) \unrhd (\rho',\lambda')$. It is then clear that $\rho^*_1 = \rho_1 = \rho'_1 = \mu_1$. Let $\mu^* = (\rho^*_2,\rho^*_3,\ldots)$. It is clear that $(\lambda^*,\rho^*)$ is a bipartition and obviously $(\lambda,\mu) \unrhd (\lambda^*,\mu^*) \unrhd (\lambda',\mu')$. Moreover if we obtained $(\mu^*,\rho^*)$ from $(\mu,\rho)$ by operations $(1)$, $(3)$ or $(4)$ respectively, we obtain $(\lambda^*,\mu^*)$ from $(\lambda,\mu)$ by operations $(2)$, $(4)$ or $(3)$ respectively.

 \end{enumerate}

\end{proof}

It is in general not true that if $\lambda$ is a partition covering $\lambda'$, then $(\lambda,\mu)$ covers $(\lambda',\mu)$, as the following example shows: 

\begin{example}\label{ex:non-example}
Consider $(\lambda,\mu) = ((3,3,2,1),(2,2,2,1))$ and $(\lambda',\mu')=((3,2,2,2),(2,2,2,1))$. Indeed, it is $\mu_1 = \mu_2 = \mu_3 > \mu_4$. Thus, there exist bipartitions which lie in between.
    \[\left(
 \ytableausetup{smalltableaux}
 \begin{ytableau}
*(white) & *(white) & *(white)\\
*(white) & *(white) & *(red) \\
*(white) & *(white) \\
*(white)
\end{ytableau} ~~~, ~~~
\ytableausetup{smalltableaux}
\begin{ytableau}
 *(white) & *(white)  \\
*(white) & *(brown)\\
*(white) &*(cyan)  \\
*(white)
\end{ytableau}  \right)
~~\unrhd ~~
\left( \ytableausetup{smalltableaux}
\begin{ytableau}
*(white) & *(white) & *(white)\\
*(white) & *(white) & *(red) \\
*(white) & *(white) \\
*(white) & *(cyan)
\end{ytableau} ~~~, ~~~
\ytableausetup{smalltableaux}
\begin{ytableau}
 *(white) & *(white)  \\
*(white) & *(brown)\\
*(white)  \\
*(white)
\end{ytableau} \right)
~~\unrhd ~~~
\left( \ytableausetup{smalltableaux}
\begin{ytableau}
*(white) & *(white) & *(white)\\
*(white) & *(white) & *(red) \\
*(white) & *(white) &*(brown) \\
*(white) & *(cyan)
\end{ytableau} ~~~, ~~~
\ytableausetup{smalltableaux}
\begin{ytableau}
 *(white) & *(white)  \\
*(white) \\
*(white) \\
*(white)
\end{ytableau} \right)
~~\unrhd ~~~
\left(\begin{ytableau}
*(white) & *(white) & *(white)\\
*(white) & *(white) \\
*(white) & *(white) \\
*(white) &  *(cyan)
\end{ytableau} ~~~, ~~~
\ytableausetup{smalltableaux}
\begin{ytableau}
 *(white) & *(white)  \\
*(white) & *(red)\\
*(white) &*(brown) \\
*(white)
\end{ytableau}\right)~. 
\]
\end{example}

Since the poset of partitions for the standard dominance order is a lattice, it is natural to ask whether this holds for our order on bipartitions. However, this is not the case already for $n=4$.
Consider $a = ((2), (1,1))$ and $b = ((2,2), \emptyset )$. Now take $c = ((2,1,1), \emptyset)$. It is covered by both $a$ and $b$, so if $a$ and $b$ have a meet, that is a greatest lower bound, it has to be $c$. However, for $d = (\emptyset, (2,2))$,   $d < a$ and $d < b$, but $c$ and $d$ are not comparable. 
Similarly, one can ask if the poset $(\BP_n,\unlhd)$ is graded, i.e., any maximal chain has equal length. However, already for $n=3$ the poset is non-graded since there exist maximal chains of length $6$ and $7$. 

\section{The posets of Specht ideals and varieties}\label{sec:inclusion}

In this section, we state and prove our main theorem:

\begin{theorem}\label{thm:equivalence}
Let $(\lambda,\mu)$ and $(\vartheta,\omega)$ be bipartitions of $n$. Let $I_{(\lambda,\mu)}, I_{(\vartheta,\omega)}$ denote their associated Specht ideals and $V_{(\lambda,\mu)}, V_{(\vartheta,\omega)}$ their associated Specht varieties. Then, the following assertions are equivalent:
\begin{enumerate}
    \item\label{equivalence_a} The bipartition $(\lambda,\mu)$ bidominates $(\vartheta,\omega)$, i.e. $(\lambda,\mu) \unrhd (\vartheta,\omega)$;
    \item\label{equivalence_b} The $B_n$-Specht ideal $I_{(\lambda,\mu)}$ contains the $B_n$-Specht ideal $I_{(\vartheta,\omega)}$, i.e. $I_{(\lambda,\mu)} \supset I_{(\vartheta,\omega)}$;
    \item\label{equivalence_c} The $B_n$-Specht variety $V_{(\lambda,\mu)}$ is contained in the $B_n$-Specht variety $V_{(\vartheta,\omega)}$, i.e. $V_{(\lambda,\mu)} \subset V_{(\vartheta,\omega)}$.
    \end{enumerate}
\end{theorem}

We start with the first implication, namely that a dominance of bipartitions implies the containment of the corresponding $B_n$-Specht ideals.
\begin{proposition}\label{prop:Specht Ideals inclusion}
Let $(\lambda,\mu), (\lambda',\mu') \in \BP_n$ be bipartitions of $n$ and let $(\lambda',\mu') \unlhd (\lambda,\mu)$. Then, $I_{(\lambda',\mu')}\subset I_{(\lambda,\mu)}$.
\end{proposition}

\begin{proof}
It is sufficient to prove the theorem in the four covering cases in Theorem~\ref{thm:Bn covering classification}. 

In cases (1) and (2), we have in particular ($\lambda' \unlhd \lambda$ and $\mu' = \mu$) or ($\mu' \unlhd \mu$ and $\lambda' = \lambda$), and the result follows from the proof of (\cite[Theorem 1]{moustrou2021symmetric}), combined with the definition of $B_n$-Specht polynomials.

Now, we consider case (3). 
In this case, to go from $(\lambda', \mu')$ to $(\lambda, \mu)$, we remove a number $a$ of boxes from a column $U_1$ in $\mu'$, that will be added to a column $U_2$ in $\lambda$. 
We can restrict our attention to these two columns. Let $b=\mid U_2\mid$, we then have $\mid U_1 \mid = a +\mid U_2 \mid = a+b$. Let $A=\{1,\ldots,a\}$, $B_1=\{a+1,\ldots,a+b\}$ and $B_2=\{a+b+1,\ldots,a+2b\}$. Up to permutation, it suffices to show that the polynomial 
\[P(\mathbf{x})=\Delta_{B_2}(\mathbf{x}^2)\Delta_{A \cup B_1}(\mathbf{x}^2) \prod_{i \in A \cup B_1}\mathbf{x}_i\] 
is in the ideal generated by polynomials of the form $\Delta_{S}(\mathbf{x}^2) \Delta_{\overline{S}}(\mathbf{x}^2) \prod_{i\in \overline{S}} \mathbf{x}_i $, where $ \{1, \ldots, a+2b\}$ is the disjoint union of $S$ and $\overline{S}$, and $\mid S \mid = a + b$.
Let us consider 
\[
Q(\mathbf{x}) = \Delta_{A \cup B_2}(\mathbf{x}^2)\Delta_{B_1}(\mathbf{x}^2) \prod_{i \in B_1} \mathbf{x}_{i},
\]
which is a polynomial of the expected form, and
\[
\begin{aligned}
\tilde{Q}(\mathbf{x}) &=  Q(\mathbf{x}) \prod_{i \in A} \mathbf{x}_i 
\\
&= \Delta_{A \cup B_2}(\mathbf{x}^2)\Delta_{B_1}(\mathbf{x}^2) \prod_{i \in A \cup B_1} \mathbf{x}_{i}. 
\end{aligned}
\]
Note that we have:
\[
\deg (P) = b(b-1) + (a+b) (a+b-1) + a+b = 2b^2 + a^2 + b(2a-1) 
\]
and 
\[
\deg (Q) = (a+b)(a+b-1) + b(b-1) +b = 2b^2 + a^2  + b(2a-1) -a 
\]
so that $P$ and $\tilde{Q}$ have the same degree. We are going to show that $P$ is a combination of $\epsilon(\sigma)\sigma \cdot \tilde{Q}$ for $\sigma$'s in $G=\mathcal{S}_{A \cup B_1}$. 
Note that we can rewrite 
\[
\Delta_{A \cup B_2}(\mathbf{x}^2) = \Delta_A(\mathbf{x}^2)\Delta_{B_2}(\mathbf{x}^2) \prod_{i \in A} R(\mathbf{x}_i^2)
\]
where 
\[
R(y) = \prod_{j\in B_2} (y-\mathbf{x}_j^2).
\]
Since $\prod_{i \in A \cup B_1} \mathbf{x}_{i}$ and $\Delta_{B_2}(\mathbf{x}^2)$ are invariant by $G$, we can factor them out and focus on the remaining terms, and we look therefore at
\[
P^*(\mathbf{x}) = \Delta_{A \cup B_1}(\mathbf{x}^2)
\]
and 

\[
Q^*(\mathbf{x}) = \Delta_A(\mathbf{x}^2)\Delta_{B_1}(\mathbf{x}^2)\prod_{i\in A} R(\mathbf{x}_i^2).
\]
Also consider the subgroup $H=\mathcal{S}_{A}\times \mathcal{S}_{B_1}$ of $G$. Then, for $\tau_1 \in \mathcal{S}_{A}$, $\tau_2 \in \mathcal{S}_{B_1}$, we have $\tau_1 \tau_2(\Delta_A(\mathbf{x}^2)) = \epsilon(\tau_1) (\Delta_A(\mathbf{x}^2))$ and 
$\tau_1 \tau_2(\Delta_{B_1}(\mathbf{x}^2)) = \epsilon(\tau_2) (\Delta_{B_1}(\mathbf{x}^2))$, and because $\prod_{i \in A} R(\mathbf{x}_i^2)$ is $H$-invariant, we get
\[
\epsilon(\tau_1 \tau_2) \tau_1 \tau_2 Q^* = Q^*,
\]
allowing us to consider the sum 
\[
\overline{Q} = \sum_{\sigma \in G/H} \epsilon(\sigma) \sigma Q^*, 
\]
and we claim that $P^* = \overline{Q}$.

First, we show that $P^*$ divides $\overline{Q}$, namely that for every $i\neq j \in A \cup B_1$, $\mathbf{x}_i^2 - \mathbf{x}_j^2$ divides $\overline{Q}$.  
Since for every $\sigma\in G$, $\sigma \overline{Q} = \pm \overline{Q}$, and $G$ acts transitively on pairs $(i,j)$, it is enough to check that $\mathbf{x}_1^2-\mathbf{x}_2^2$ divides $\overline{Q}$. 
We hence have to show that $\overline{Q}$ vanishes when imposing $\mathbf{x}_1^2 = \mathbf{x}_2^2$.
To see this, first observe that the terms in the sum are in correspondence with set partitions $K \cup \overline{K}$ of $A \cup B_1$, where $\mid K \mid = a$. Indeed, up to permutation by elements of $H$, we only need to choose where to send the subset $A=\{1, \ldots, a\}$.
Now, if $\sigma$ sends $1$ and $2$ in the same subset, the corresponding Vandermonde determinant in $\epsilon (\sigma) \sigma Q^*$ vanishes whenever $\mathbf{x}_1^2=\mathbf{x}_2^2$.
We then only need to focus on partitions where $1$ and $2$ are not in the same subset. 
There are two kinds of such partitions: those with $1\in K$ and $2 \in \overline{K}$, and those $2\in K$ and $1 \in \overline{K}$. The transposition $(12)$ naturally induces a bijection between these sets of partitions. If $\sigma \in G$ is a representative for a partition of the first kind, then $(1 2)\sigma $ is a representative for the corresponding partition of the second kind.
When $\mathbf{x}_1^2=\mathbf{x}_2^2$, we have $(12) \sigma Q^* (\mathbf{x}) = \sigma Q^*(\mathbf{x})$, and because $\epsilon ((12) \sigma) = - \epsilon(\sigma)$, the two corresponding terms cancel out. 

Then, we need to check that $P^*$ and $\overline{Q}$ have the same leading term with respect to the lexicographical ordering.   
The leading term of $P^*$ is $\mathbf{x}_1^{2(a+b-1)}\mathbf{x}_2^{2(a+b-2)} \cdots \mathbf{x}_{a+b-1}^2$. Then, for $\sigma \in G$ sending $\{1, \ldots, k\} $ onto $K$, the partial degree of $\sigma  Q^*$ in $\mathbf{x}_1$ is
\[
\begin{cases}
2(a-1) + 2b & \text{ if } 1 \in K
\\
2(b-1)  & \text{ if } 1 \notin K
\end{cases}
\]
and therefore $\sigma  Q^*$ can give a contribution to the leading term of $\overline{Q}$ only if $1 \in K$. By the same argument, $2$ has to be in $K$, and in the end, $K = \{1, \ldots, a\}$: Indeed, assume there is a minimal $i \leqslant a$ with $i \notin K$. Then, the leading term of $\sigma  Q^*$ is of the form $\mathbf{x}_1^{2(a+b-1)}\mathbf{x}_2^{2(a+b-2)}\cdots \mathbf{x}_{j-1}^{2(a+b-j+1)} \mathbf{x}_j^{2(b-1)} m$ where $m$ is a monomial in the variables $\mathbf{x}_{j+1},\ldots,\mathbf{x}_{a+b}$. Since, $j\leq a$, then $2(b-1) < 2(a+b-j)$, and therefore the leading monomial of $\sigma  Q^*$ is strictly lower than that of $Q^*$.
Thus, the leading term of $\overline{Q}$ is exactly the leading term of $Q^*$, which is $\mathbf{x}_1^{2(a+b-1)}\mathbf{x}_2^{2(a+b-2)} \cdots \mathbf{x}_{a+b-1}^2$, as expected. This concludes the covering case (3).

The proof for the covering case (4) is very similar. In this situation, to go from $(\lambda', \mu')$ to $(\lambda, \mu)$, we remove $a$ boxes from a column $U_1$ in $\lambda'$, before adding them to a column $U_2$ in $\mu$, with $\mid U_2 \mid = b$ and $\mid U_1 \mid = a+b+1 $.
We can apply the previous argument, where this time $A = \{1, \ldots, a\}$, $B_1 = \{a+1, \ldots, a+b+1 \}$, $B_2=\{a + b + 2, \ldots, a+2b+1\}$, 
\[
P = \Delta_{A \cup B_1}(\mathbf{x}^2)\Delta_{B_2}(\mathbf{x}^2) \prod_{i \in  B_2}\mathbf{x}_i
\]
\[
Q =\Delta_{B_1}(\mathbf{x}^2)\Delta_{A \cup B_2}(\mathbf{x}^2)\prod_{i \in A \cup B_2}\mathbf{x}_i 
\]
 and
\[
\begin{aligned}
\tilde{Q} &=   (\prod_{j \in A} \mathbf{x}_{j}) Q 
\\
&= \Delta_{B_1}(\mathbf{x}^2)\Delta_{A \cup B_2}(\mathbf{x}^2)\prod_{i \in A}\mathbf{x}_i^2 \prod_{i \in B_2}\mathbf{x}_i. 
\end{aligned}
\]

\end{proof}

The second implication ((2) implies (3)) of Theorem~\ref{thm:equivalence} is clear, it remains to prove that (3) implies (1):

\begin{proposition} \label{prop:varieties and dominance order}
Let $(\lambda,\mu),(\vartheta,\omega)$ be bipartitions and $V_{(\lambda,\mu)} \subset V_{(\vartheta,\omega)}$. Then, $(\lambda,\mu) \unrhd (\vartheta,\omega)$.
\end{proposition}
To prove this implication, we will consider two types of points in $\K^n$:
\begin{lemma}\label{lem:z1}
Let $(\vartheta,\omega) \in \BP_n$ and $\Lambda = \vartheta \uplus \omega$ be a partition of $n$. Consider the point
\[
z  = (\underbrace{a_1,\ldots,a_1}_{\Lambda_1},\underbrace{a_2,\ldots,a_2}_{\Lambda_2},\ldots,\underbrace{a_m,\ldots,a_m}_{\Lambda_{m}})
\] with $a_i^2 \neq a_j^2$ if $i \neq j$ and $a_i \neq 0$ if $i\leqslant \len(\omega)$.

\begin{enumerate}
    \item[i)]  $z \notin V_{(\vartheta,\omega)}$.
    \item[ii)] If $(\lambda, \mu) \in \BP_n$ is a bipartition such that $z \notin V_{(\lambda, \mu )}$, then $\lambda \uplus \mu \trianglerighteq \Lambda$.
\end{enumerate}

\end{lemma}

\begin{proof}

\begin{enumerate}[leftmargin=11pt]
    \item[i)] Let $(T,S)$ be the generalized bitableau of shape $(\vartheta,\omega)$ which has the filling 

\[ \left(\label{eq:bitableaux}\ytableausetup{boxsize=1.8em}
\begin{ytableau}
        a_1 & a_1 & \none[\cdots] & a_1 \\
         a_2 & a_2 & \none[\cdots]& \none \\
         \none[\vdots] & \none & \none & \none  \\
        \none[\vdots] & \none & \none & \none  \\
         a_h &\none[\cdots]
\end{ytableau}     \; \hspace{1cm} , \hspace{1cm} \;\begin{ytableau}
        a_1 & a_1 & \none[\cdots]& \none[\cdots] & a_1 \\
         a_2 & a_2 & \none[\cdots]& \none[\cdots]& \none[] \\
        \none[\vdots] & \none & \none & \none &\none[] \\
         a_l & \none[\cdots] & \none & \none\\
         \none[]
\end{ytableau} 
\right), \]
i.e., the $i$-th row of both $T$ and $S$ contains only $a_i$'s.
The assumption $a_i\neq 0$ for $i\leqslant \len(\omega)$ ensures that $S$ contains no $0$ entry, and by construction the squares of column entries are pairwise different. Thus, $z \in V_{(\vartheta,\omega)}^c$.
    \item[ii)] By assumption, there is a bitableau $(T,S)$ of shape $(\lambda, \mu)$ such that $\spe_{(T,S)}(z) \neq 0$. Then, according to Lemma~\ref{lemma:spechtofsums}, there is a tableau $U = T \uplus S$ of shape $\lambda \uplus \mu$ such that $\spe_U(z^2)\neq 0$. Therefore $z^2$ does not belong to the $\mathcal{S}_n$-Specht variety $V_{\lambda \uplus \mu}$, and since $z$ and $z^2$ have the same $S_n$-orbit type, (\cite[Prop 1.ii)]{moustrou2021symmetric}) gives
\[
\Lambda = (\Lambda_1,\ldots,\Lambda_m)  \trianglelefteq \lambda \uplus \mu, 
\]
which proves the lemma.
\end{enumerate}

\end{proof}

\begin{lemma}\label{lem:z2}
Let $(\vartheta,\omega) \in \BP_n$ be a bipartition of $n$. Let $m = \max\{\len(\vartheta), \len(\omega)\}$. Consider the point
\[
z = (\underbrace{0,\ldots,0}_{\vartheta_1},\underbrace{a_1,\ldots,a_1}_{\omega_1+\vartheta_2},\underbrace{a_2\ldots,a_2}_{\omega_2+\vartheta_3},\ldots,\underbrace{a_m,\ldots,a_m}_{\omega_m+\vartheta_{m+1}})
\] with $a_i^2 \neq a_j^2$ if $i \neq j$ and $a_i \neq 0$.
Then:
\begin{itemize}
    \item[i)]  $z \notin V_{(\vartheta,\omega)}$.
    \item[ii)] If $(\lambda, \mu) \in \BP_n$ is a bipartition such that $z \notin V_{(\lambda, \mu )}$, then \[ \sum_{j=1}^{k-1} (\lambda_j + \mu_j) + \lambda_{k} \geq   \sum_{j=1}^{k-1} (\vartheta_j + \omega_j) + \vartheta_{k}\]
for any integer $k \geq 1$.
\end{itemize}

\end{lemma}

\begin{proof}

\begin{enumerate}[leftmargin=11pt]
    \item[i)] Let $(T,S)$ be the generalized bitableau of shape $(\vartheta,\omega)$ which has the filling \[\left(\label{eq:bitableaux2}
\begin{ytableau}
        0 & 0 & \none[\cdots] & 0 \\
         a_1 & a_1 & \none[\cdots]& \none \\
        \none[\vdots] & \none & \none & \none  \\
         a_h & \none[\cdots] & \none & \none\\
         \none[]
\end{ytableau} \hspace{1cm} , \hspace{1cm} 
\begin{ytableau}
        a_1 & a_1 & \none[\cdots] & a_1 \\
         a_2 & a_2 & \none[\cdots]& \none \\
         \none[\vdots] & \none & \none & \none  \\
        \none[\vdots] & \none & \none & \none  \\
         a_l & \none[\cdots] & \none & \none,
\end{ytableau}    
\right)\]
i.e., $\vartheta_i$ contains only $a_{i-1}$'s and $\omega_i$ contains only $a_i$'s, where $a_0 = 0$.
We observe that no entry in $S$ equals $0$ and the squares of column entries are pairwise different. Thus, we have $z \in V_{(\vartheta,\omega)}^c$.
    \item[ii)] By assumption, there exists a bitableau $(T,S)$ of shape $(\lambda,\mu)$ such that \[0 \neq \spe_{(T,S)}(z) = \spe_T(z^2)\spe_S(z^2) \cdot \prod_{j \in S} z_j.\] Let $(T^*,S^*)$ be the generalized bitableau obtained from $(T,S)$ by replacing $i$ with $z_i$ in any box. This means that the zeros of $z$ are written in $T^*$ and no column in $T^*$ or $S^*$ contains entries with equal squares. Since permutation of the column entries can only change the sign of $\spe_{(T,S)}(z)$, we can assume that the entries in every column in $(T^*,S^*)$ are sorted increasingly by the indices of the $a_i$'s from above to below, and with $a_0=0$. \\
We obtain that all the $0$'s must be written in the first row of $T^*$ which implies $\lambda_1 \geq \vartheta_1$. Now, for an integer $k \geq 1$ the $a_k$'s in $(T^*,S^*)$ must be written in different columns in the generalized bitableau $(T^*,S^*)$. Since the entries in $(T^*,S^*)$ are written with increasing indices in each column from the top to the bottom, we know that the $a_j$'s with $0 \leqslant j \leqslant k$ must be written within the first $k$ rows of $S$ and the first $k+1$-rows in $T$. Thus, by the pigeon hole principle, we have \[ 
\sum_{j=1}^k (\lambda_j + \mu_j) + \lambda_{k+1} \geq   \sum_{j=1}^k (\vartheta_j + \omega_j) + \vartheta_{k+1} .\]
\end{enumerate}

\end{proof}

Now, we can prove Proposition~\ref{prop:varieties and dominance order}:

\begin{proof}[Proof of Proposition~\ref{prop:varieties and dominance order}]
The assumption is equivalent to $V_{(\vartheta,\omega)} ^c \subset  V_{(\lambda,\mu)} ^c$.
We have to prove that $\lambda \uplus \mu \trianglerighteq \vartheta \uplus \omega$ and that $\sum_{j=1}^{k-1} (\lambda_j + \mu_j) + \lambda_{k} \geq   \sum_{j=1}^{k-1} (\vartheta_j + \omega_j) + \vartheta_{k}$ for every integer $k \geq 1$.

For the first claim, consider the point
\[
z = (\underbrace{a_1,\ldots,a_1}_{\vartheta_1 + \omega_1},\underbrace{a_2,\ldots,a_2}_{\vartheta_2 + \omega_2},\ldots,\underbrace{a_m,\ldots,a_m}_{\vartheta_m + \omega_m})
\]
with $a_i^2 \neq a_j^2$ for $i\neq j$, and $a_i \neq 0$ if $i \leq \len(\omega)$.
According to i) in Lemma~\ref{lem:z1}, $z \in  V_{(\vartheta,\omega)} ^c$. By assumption, we then have $z \in  V_{(\lambda,\mu)} ^c$, and ii) in Lemma~\ref{lem:z1} gives 
\[
\lambda \uplus \mu \trianglerighteq \Lambda(z) = \vartheta \uplus \omega.
\]
For the second claim, consider the point
\[
z = (\underbrace{0,\ldots,0}_{\vartheta_1},\underbrace{a_1,\ldots,a_1}_{\omega_1+\vartheta_2},\underbrace{a_2\ldots,a_2}_{\omega_2+\vartheta_3},\ldots,\underbrace{a_m,\ldots,a_m}_{\omega_m+\vartheta_{m+1}})
\]
with $a_i^2 \neq a_j^2$ for $i\neq j$, and $a_i \neq 0$.
According to i) in Lemma~\ref{lem:z2}, $z \in  V_{(\vartheta,\omega)} ^c$. By assumption, we then have $z \in  V_{(\lambda,\mu)} ^c$, and ii) in Lemma~\ref{lem:z2} gives 
 \[ \sum_{j=1}^{k-1} (\lambda_j + \mu_j) + \lambda_{k} \geq   \sum_{j=1}^{k-1} (\vartheta_j + \omega_j) + \vartheta_{k}\]
for any integer $k \geq 1$.
\end{proof}

\section{Orbit types}\label{sec:varieties}

In this section we define orbit types of elements in $\K^n$ with respect to the action of the hyperoctahedral group. Compared with the $\mathcal{S}_n$-orbit types, they allow a finer set decomposition of $\K^n$ since one distinguishes whether coordinates are $0$ or not. This leads to a set partition of the $B_n$-Specht varieties based on the combinatorics of the poset $(\BP_n,\unlhd)$. 

Recall that if $z=(a_1,\ldots,a_n) \in \K^n$, then the $\mathcal{S}_n$-orbit type of $z$ is the unique partition $\Lambda(z) = (\Lambda_1,\ldots,\Lambda_l) \vdash n$ such that $\Stab_{\mathcal{S}_n}(z) \simeq \Z/\Lambda_1\Z \times \cdots \times \Z/\Lambda_l\Z$, or equivalently, there exists $b_1,\ldots,b_l$ pairwise distinct such that $z \in \mathcal{S}_n \cdot (\underbrace{b_1,\ldots,b_1}_{\Lambda_1},\ldots, \underbrace{b_l,\ldots,b_l}_{\Lambda_l})$.

\begin{definition} \label{def:t-cut}
Let $\lambda=(\lambda_1,\ldots,\lambda_m) \vdash n$ be a partition and $t \in \N_0$. Let $j=\min\{i : \lambda_i<t\}$ with the convention that $j= m+1$ if $t=0$. Then the \emph{$t$-cut} of $\lambda$ is the bipartition $(\rho,\sigma)$ defined as $\rho=(t,\cdots,t,\lambda_j,\cdots,\lambda_m)$ and  $\sigma=(\lambda_1-t,\cdots,\lambda_{j-1}-t)$. We denote it by $\cut (\lambda,t)$.
\end{definition}

We have $\cut (\lambda,0) = (\emptyset,\lambda)$ while $\cut (\lambda,t) = (\lambda,\emptyset)$ for any $t \geqslant \lambda_1$.
We observe that if $\cut(\lambda,t)=(\sigma, \rho)$, then $\sigma \uplus \rho = \lambda$.

We are now ready to define the $B_n$-orbit type of a point in $\K^n$ and the notion of $B_n$-orbit set:

\begin{definition} \label{def:orbittypes}
 Let $z \in \K^n$.  The \emph{$B_n$-orbit type} of $z$ is \[\Omega(z) = \cut(\Lambda(z^2),t_z)\] where $t_z$ is the number of $0$ coordinates of $z$. 

 For $(\lambda,\mu) \in \BP_n$, the \emph{$B_n$-orbit set} associated to $(\lambda,\mu)$ is then \[H_{(\lambda,\mu)} =\left\{z \in \K^n :\ \Omega(z) = (\lambda,\mu)\right\}.\]
\end{definition}

We observe that $B_n$-orbits sets might be empty. The non-empty $B_n$-orbits sets correspond to bipartitions $(\lambda,\mu)$ such that $\lambda_1= \cdots=\lambda_{\len(\mu)+1}.$ Moreover, if $(\lambda,\mu)$ is such that $H_{(\lambda,\mu)} \neq \emptyset$, then any point $z \in H_{(\lambda,\mu)}$ is of the following form: let $m=\len(\mu)$. Then there exist non-zero elements $a_1,\ldots,a_l \in \K$, with distinct squares such that \[z = \sigma \cdot (\underbrace{a_1,\ldots,a_1}_{\lambda_1+\mu_1},\ldots,\underbrace{a_{m},\ldots,a_{m}}_{\lambda_1+\mu_m},\underbrace{0,\ldots,0}_{\lambda_1},\underbrace{a_{m+1},\ldots,a_{m+1}}_{\lambda_{m+2}},\ldots,\underbrace{a_l,\ldots,a_l}_{\lambda_{l+1}})\] for some $\sigma \in B_n$. It is straightforward that $l=m$ if $\len(\lambda)\leqslant m$ (which implies $\lambda = \emptyset$) and  $l=\len(\lambda) -1$ if $\len(\lambda)> m$.
With such a point of view, this notion of $B_n$-orbit type is a natural generalization of $\mathcal{S}_n$-orbit type in terms of hyperplane arrangements. While $\mathcal{S}_n$-orbit types correspond with unions of intersections of hyperplanes of type $A_{n-1}$ of the form $x_i-x_j = 0$, $B_n$-orbit types correspond with unions of intersections of hyperplanes of type $B_n$, of the form $x_i\pm x_j = 0$ and $x_i = 0$. 

\begin{example}
We present the orbit types of $\K^3$.  Let $a,b,c \in \K^*$ be such that they have distinct squares.
    \begin{center}
\begin{tabular}{ c|c } 
$z$ & $\Omega(z)$  \\
 \hline
 $(0,0,0)$ &  $((3),\emptyset)$ \\
$(\pm a,0,0)$ &  $((2,1),\emptyset)$ \\
$(\pm a, \pm b,0)$ & $((1,1,1),\emptyset)$  \\
 $(\pm a, \pm a,0)$ & $((1,1),(1))$ \\
$(\pm a,\pm a, \pm a)$ & $(\emptyset,(3))$  \\
$(\pm a,\pm a,\pm b)$ & $(\emptyset,(2,1))$  \\
$(\pm a,\pm b,\pm c)$ & $(\emptyset,(1,1,1))$ 
\end{tabular}
\end{center}
The remaining bipartitions $((2),(1)),((1),(2)),((1),(1,1))$ of $3$ have an empty orbit set.
\end{example}

The following proposition follows from the previous definitions and comments:

\begin{proposition}\label{prop:from Bn to Sn orbit type}
 Let $z \in \K^n$ and $(\lambda,\mu)=\Omega(z)$. Then $\Lambda(z^2)= \lambda \uplus \mu$.
 Moreover, the $B_n$-orbit sets define a set partition of $\K^n$, namely $\K^n = \biguplus_{(\lambda,\mu) \in \BP_n} H_{(\lambda,\mu)}$.
\end{proposition}

\begin{proposition}\label{prop:not in variety and dominance}
Let $z \in \K^n$, and $(\lambda,\mu)  \in \BP_n$. Then: \begin{enumerate}
    \item $z \not \in V_{\Omega(z)}$,
    \item $z \not \in V_{(\lambda,\mu)} \Rightarrow (\lambda,\mu) \unrhd \Omega(z)$.
\end{enumerate}
\end{proposition}

\begin{proof}
Let $(\vartheta,\omega) = \Omega(z)$ and $m = \len(\omega)$. Then there exist $\sigma \in B_n$ and $a_1,\ldots,a_l \in \K^* $ with distinct squares such that $z = \sigma z'$ with \[z' = (\underbrace{a_1,\ldots,a_1}_{\vartheta_1+\omega_1},\ldots,\underbrace{a_{m},\ldots,a_{m}}_{\vartheta_1+\omega_m},\underbrace{0,\ldots,0}_{\vartheta_1},\underbrace{a_{m+1},\ldots,a_{m+1}}_{\vartheta_{m+2}},\ldots,\underbrace{a_l,\ldots,a_l}_{\vartheta_{l+1}}).\] Since the Specht varieties are invariant under the action of $B_n$, we can assume that $z=z'$.
With this shape, we can apply Lemma~\ref{lem:z1} to $z$ which gives immediately (1), and partly (2): if $z\notin V_{(\lambda, \mu)}$, then $\lambda \uplus  \mu \trianglerighteq \vartheta \uplus \omega$.
It remains to prove that if $z \not\in V_{(\lambda, \mu)}$, then $ \sum_{j=1}^{k-1} (\lambda_j + \mu_j) + \lambda_{k} \geq   \sum_{j=1}^{k-1} (\vartheta_j + \omega_j) + \vartheta_{k}$ for any integer $k \geq 1$.
To do so, it is enough to observe that since $\vartheta_1=\cdots=\vartheta_{m+1}$, the point 
\[z'' =  (\underbrace{0,\ldots,0}_{\vartheta_1},\underbrace{a_1,\ldots,a_1}_{\omega_1+\vartheta_2},\underbrace{a_2\ldots,a_2}_{\omega_2+\vartheta_3},\ldots,\underbrace{a_m,\ldots,a_m}_{\omega_m+\vartheta_{m+1}},\underbrace{a_{m+1},\ldots,a_{m+1}}_{\vartheta_{m+2}},\ldots,\underbrace{a_l,\ldots,a_l}_{\vartheta_{l+1}})\]
is in the same orbit, and we can apply Lemma~\ref{lem:z2} to conclude the proof.

\end{proof}

As a consequence of our previous results, we get a decomposition of $B_n$-Specht varieties in terms of orbit sets:

\begin{theorem} \label{thm: specht varieties and orbit types}
$$V_{(\lambda,\mu)} = \left( \bigcup_{(\vartheta,\omega)\in \BP_n, (\vartheta,\omega) \unlhd (\lambda,\mu)} H_{(\vartheta,\omega)} \right)^c = \bigcup_{(\vartheta,\omega) \in \BP_n, (\vartheta,\omega) \not \unlhd (\lambda,\mu)} H_{(\vartheta,\omega)}.$$
\end{theorem}
\begin{proof}
The collection $\left\{ H_{(\vartheta,\omega)} :(\vartheta,\omega) \in \BP_n  \right\}$ defines a set partition of $\K^n$ by definition, which explains the second equality. \\
In order to prove the first one, we first assume that $z \not \in V_{(\lambda,\mu)}$. By part (2) in Proposition \ref{prop:not in variety and dominance} we obtain that $(\lambda,\mu) \unrhd \Omega (z)$. Thus $z \in \bigcup_{(\vartheta,\omega)\in \BP_n, (\vartheta,\omega) \unlhd (\lambda,\mu)} H_{(\vartheta,\omega)} $.\\
Conversely, let $z \in \bigcup_{(\vartheta,\omega) \unlhd (\lambda,\mu)} H_{(\vartheta,\omega)}.$ 
In other words, $\Omega(z) \unlhd (\lambda,\mu)$.
Then, part (1) in Proposition~\ref{prop:not in variety and dominance} implies $z \not \in V_{\Omega(z)}$. On the other hand, by Theorem~\ref{thm:equivalence}, $V_{(\lambda,\mu)} \subset V_{\Omega(z)} $. Therefore $z \not \in V_{(\lambda,\mu)}$.
\end{proof}
\begin{example} \label{example:orbit types}
We calculate the Specht variety corresponding with the bipartition $((1,1),(2))$ using Theorem~\ref{thm: specht varieties and orbit types}. The bipartitions $(\vartheta, \omega)$ encoding non-empty orbit sets such that $(\vartheta,\omega) \not \unlhd ((1,1),(2))$ are the bipartitions in 
\[\Omega = \{((2,2),\emptyset),((2,1,1),\emptyset),(\emptyset,(4)),((3,1),\emptyset),((4),\emptyset)\} \subset \BP_4.\]
Then,
\begin{align*}
     V_{((1,1),(2))}  & = \bigcup_{(\lambda,\mu) \in \Omega} H_{(\lambda,\mu)}\\
      & = H_{((2,2),\emptyset)} \cup H_{((2,1,1),\emptyset)}  \cup H_{(\emptyset,(4))} \cup H_{((3,1),\emptyset)} \cup H_{((4),\emptyset)},
\end{align*}
which means
\begin{align*}
    V_{((1,1),(2))} & = B_4 \cdot \{(0,0,a,a),(0,0,a,b),(a,a,a,a),(0,0,0,a),(0,0,0,0) : a,b \in \K_{>0}\}.
\end{align*} 
\end{example}

One might look for a more natural orbit-type, only involving the number of zeroes of a point, and the $\mathcal{S}_n$-orbit-type of the remaining non-zero squared coordinates. 
Indeed, our previous decomposition can be reformulated in such a way, and it can be  obtained either using $S_n$-invariance and results of \cite{moustrou2021symmetric}, or as a consequence of our previous results on bipartitions. 
We just briefly describe here the latter approach because such a point of view, even if it gives a natural decomposition, does not give information on inclusions of $B_n$-Specht varieties, which will be needed for our applications in the next section. 

If $z = (a_1,\ldots,a_n) \in \K^n$ is a point, then \[\Stab_{\mathcal{S}_n}(a_1^2,\ldots,a_n^2) \simeq \mathcal{S}_{\Lambda_1}\times\cdots\times \mathcal{S}_{\Lambda_l}\times \mathcal{S}_t\] where $\Lambda_1\geqslant \ldots\geqslant \Lambda_l$ and $t$ is the number of zero coordinates of $z$. We could have defined the orbit type of $z$ as \[\Lambda(z) = (t,(\Lambda_1,\ldots,\Lambda_l)).\] 

Then, there is a bijection $\varphi$ between the set of pairs $(t,\Lambda)$ where $n\geqslant t\geqslant 0$ and $\Lambda \vdash n-t$ and the set of bipartitions $(\lambda,\mu)$ such that $H_{(\lambda,\mu)} \neq \emptyset$ given by \[(t,\Lambda) \mapsto \cut ( \overline{(\Lambda_1,\Lambda_2,\ldots,\Lambda_l,t)},t)\] where $\overline{(\Lambda_1,\Lambda_2,\ldots,\Lambda_l,t)}$ denotes the partition obtained by rearranging $(\Lambda_1,\Lambda_2,\ldots,\Lambda_l,t)$ in non increasing order.

Now, for $t \leqslant 0$ and $\Lambda \vdash n-t$, if we denote by $G_{t,\Lambda} = \{z \in \K^n;\ \Lambda(z) = (t,\Lambda)\}$, we have by construction 
\[
G_{t, \Lambda} = H_{\phi(t, \Lambda)}.
\]
Moreover, $\phi$ preserves the orders in the following sense: for $t$ fixed, $\phi(t, \Lambda) \trianglelefteq \phi(t, \Lambda')$ in our poset of bipartitions if and only if $\Lambda \trianglelefteq \Lambda'$ in the poset of partitions. Also, it is obvious that if $t> \lambda_1$, then $z \in V_{(\lambda,\mu)}$. 

As a consequence, our decomposition in Theorem~\ref{thm: specht varieties and orbit types} becomes in this context
\[
\left(V_{(\lambda,\mu)}\right)^c = \bigcup_{t=1}^{\lambda_1} \bigcup_{\substack{\Lambda, \\  \phi(t,\Lambda) \unlhd (\lambda, \mu)}} G_{t,\Lambda}.
\]

Actually, if one fixes  $0\leqslant t \leqslant \lambda_1$, one can prove that  \[\phi(t,\Lambda) \unlhd (\lambda, \mu) \Leftrightarrow \Lambda  \unlhd \lambda^{(t)} \uplus \mu,\] where $\lambda^{(t)}$ is defined as follows: if $s= \max\{i;\ \lambda_i \geqslant t\}$, then \[\lambda^{(t)} = (\lambda_1,\ldots,\lambda_{s-1},\lambda_s+\lambda_{s+1}-t,\lambda_{s+2},\ldots,\lambda_{l}).\]

Finally, we can reformulate the decomposition as: 
\[\left(V_{(\lambda,\mu)}\right)^c = \bigcup_{t=1}^{\lambda_1} \bigcup_{\Lambda \unlhd \lambda^{(t)} \uplus \mu} G_{(t,\Lambda)}.\]

\section{Applications to \texorpdfstring{$B_n$}{Bn}-invariant ideals}\label{sec:genideals}
\subsection{Specht ideals in $B_n$-invariant ideals}
The main result in the article~\cite{moustrou2021symmetric} studies the $\mathcal{S}_n$-Specht ideals contained in an ideal $I$ which is globally invariant under the action of $\mathcal{S}_n$. We give here the main ideas of this statement, before extending it to the case of $B_n$.

For $P$ a polynomial in an $\mathcal{S}_n$-invariant ideal $I \subset \K[\mathbf{x}_1, \ldots, \mathbf{x}_n]$ of degree $d$, we denote by $\wt(P_d)$ the number of variables appearing in the highest component $P_d$ of $P$. Moreover, for a monomial $m$ of degree $d$ in $P$, the partial degrees of $m$ induce a partition $(k_1, \ldots, k_\ell)$, and under the assumption $\wt(P_d) + d \leq n$, we have in particular $d + \ell \leq n$, and therefore we can define the partition
\[
\mu(m)=(k_1 +1, k_2 + 1, \ldots, k_\ell + 1, \underbrace{1, \ldots, 1}_{n-d-\ell}) 
\]
of $n$.
It is then proved (\cite[Theorem 4]{moustrou2021symmetric}) that for every monomial $m \in \Mon(P_d)$, the ideal $I$ contains every $\spe_T$ for which  $\sh(T) \trianglelefteq \mu(m)^\perp$.

The proof works as follows: up to permutation of the variables, we may assume that 
\[m=\mathbf{x}_1^{k_1}\mathbf{x}_2^{k_2}\cdots \mathbf{x}_\ell^{k_\ell}\]
and since $\wt(P_d)+ d \leq n$, there exists $d = k_1 + \ldots + k_\ell $ many variables in $\{\mathbf{x}_1, \ldots \mathbf{x}_n\}$ that do not appear in $P_d$. 
More precisely, we can take $I_1, \ldots, I_\ell$, disjoint subsets of $\{1,\ldots,n\}$ such that for any $1\leq i \leq \ell$, there are $k_i$ elements in $I_i$, and none of them appears in $P_d$. 
Then, if for $1\leq i \leq \ell$, $J_i = \{i\} \cup I_i$, we can prove
\begin{equation}\label{eq:DeltaSn}
\Delta_{J_1}\cdots\Delta_{J_\ell} = \frac{k}{ k_1!\cdots k_\ell!} \sum_{\sigma \in \Sym_{J_1}\times \cdots \times\Sym_{J_\ell}} \epsilon(\sigma) \sigma (\Delta_{I_1}\cdots\Delta_{I_\ell}P)
\end{equation}
where $\Delta_{I}$ is the Vandermonde polynomial of the ordered set $I$ and $k \neq 0$. 

Now, we want to generalize this result to $B_n$-invariant ideals. First, we need to associate a bipartition to a given monomial:
\begin{definition}\label{def:Gamma(m)}
Let $m$ be a monomial in $\K[\mathbf{x}_1, \ldots, \mathbf{x}_n]$. There exist unique sets $I_1$ and $I_2$ such that we can write $m$ as 
\[
m = \prod_{i\in I_1} \mathbf{x}_i^{2k_i}\prod_{i\in I_2} \mathbf{x}_i^{2r_i} \prod_{i\in I_2} \mathbf{x}_i
\]
and $k_i \neq 0$.
Denote $\ell=\mid I_1 \mid$, $d_1 = \sum_{i\in I_1} k_i$, $s = \mid I_2 \mid$, and $d_2 = \sum_{i\in I_2} r_i$. 
The sets $\{k_i, i \in I_1\}$ and $\{r_i, i \in I_2\}$ respectively induce partitions $(\lambda_1, \ldots, \lambda_\ell)$ of $d_1$ and $(\mu_1, \ldots, \mu_s)$ of $d_2$.
If moreover we assume that $ \ell + s + d_1 + d_2 \leq n$, we can define a bipartition $\Gamma(m)$ of $n$ by
\[
\Gamma(m)=(\tilde{\lambda}, \tilde{\mu})=((\lambda_1 +1, \lambda_2 + 1, \ldots, \lambda_\ell + 1, \underbrace{1, \ldots, 1}_{n-(\ell + s + d_1 + d_2)}),(\mu_1 +1, \mu_2 + 1, \ldots, \mu_s + 1)).
\]
Finally, we define 
\[
\Gamma^*(m) = ( \tilde{\lambda}^\perp,\tilde{\mu}^\perp) = (n-(s+d_1+d_2),\ell,\tilde{\lambda}^\perp_3,\ldots),(s,\tilde{\mu}_2^\perp,\ldots)),
\]
which is a bipartition of $n$ as well.
\end{definition}

With this notion, we get, for $B_n$-invariant ideals:

\begin{theorem}\label{thm:Ideals}
Let $I \subset \K[\mathbf{x}_1,\ldots,\mathbf{x}_n]$ be a $B_n$-invariant ideal, and let $P \in I$.
Assume that $m$ is a monomial in the homogeneous component of highest degree of $P$. Using the notation of Definition~\ref{def:Gamma(m)}, assume that  $\wt(P_d) + d_1 + d_2 \leqslant n$. Then, we have the ideal inclusion
\[
I_{\Gamma^*(m)} \subset I.
\]
\end{theorem}

\begin{proof}
Up to permutation, we may assume that 
\[
m = \prod_{i=1}^\ell \mathbf{x}_i^{2k_i}\prod_{i=\ell+1}^{\ell + s} \mathbf{x}_i^{2r_i} \prod_{i=\ell+1}^{\ell + s} \mathbf{x}_i
\]
and that the coefficient of $m$ in $P$ is $1$.
Let $\epsilon_i \in B_n$ the map changing $\mathbf{x}_i$ in $-\mathbf{x}_i$. Then, the polynomial
\[
\frac{P - \epsilon_i P}{2}
\]
is a polynomial in $I$ whose terms are exactly the terms of $P$ having an odd degree in $\mathbf{x}_i$, and therefore divisible by $\mathbf{x}_i$. 
After applying this transformation for every $i\in \{\ell + 1, \ldots, \ell +s\}$, we may substitute $P$ with a polynomial of the form
\[
\tilde{P}(\mathbf{x})\prod_{i=\ell+1}^{\ell + s} \mathbf{x}_i ,
\] 
containing $m$ in its leading term, and where every term in $\tilde{P}$ has even degree in $\mathbf{x}_i$, for every $i\in \{\ell + 1, \ldots, \ell +s\}$.
Further, for $i\notin \{\ell + 1, \ldots, \ell +s\}$, we can apply the transformation
\[
\frac{\tilde{P}(\mathbf{x}) + \epsilon_i {\tilde{P}(\mathbf{x})}}{2}
\]
to get a polynomial which is still in $I$, but its terms are exactly the terms of $P$ having an even degree in $\mathbf{x}_i$. 
In the end, we may assume that $P$ is of the form:
\[
P ( \mathbf{x}) = Q(\mathbf{x}^2) \prod_{i=\ell+1}^{\ell + s} \mathbf{x}_i ,
\]
where $m$ is still a monomial of the leading term. 

Now we can apply a strategy similar to the one described for $\mathcal{S}_n$-invariant ideals. 
Since $\ell + s + d_1 + d_2 \leq \wt(P_d)+ d_1 + d_2 \leq n$, there exists $d_1 +d_2 = k_1 + \cdots + k_\ell + r_1 + \cdots + r_s $ many variables in $\{\mathbf{x}_1, \ldots \mathbf{x}_n\}$ that do not appear in $P_d$, and we can take $I_1, \ldots, I_\ell$, $I_{\ell+1}, \ldots, I_{\ell+s}$, disjoint subsets of $\{1,\ldots,n\}$ such that for any $1\leq i \leq \ell$, there are $k_i$ elements in $I_i$, for any $\ell + 1\leq i \leq \ell +s$, there are $r_i$ elements in $I_i$, and none of them appears in $P_d$. 
Then, for $1\leq i \leq \ell + s$, denote 
\[
J_i = \{i\} \cup I_i,
\]
and $\tilde{\Delta}_J(\mathbf{x})= \Delta_J(\mathbf{x}^2)$ the Vandermonde polynomial associated with the variables $\mathbf{x}_i^2$ for $i \in J$.
Consider 
\[
R(\mathbf{x})= P(\mathbf{x})\tilde{\Delta}_{I_1}\cdots\tilde{\Delta}_{I_{\ell+s}}\prod_{i \in {I_{\ell+1}}\cup \cdots \cup {I_{\ell+s}}}\mathbf{x}_i
\]
We then have
\[
\begin{aligned}
\sum_{\sigma \in \Sym_{J_1}\times \cdots \times\Sym_{J_{\ell+s}}} \epsilon(\sigma) \sigma (R(\mathbf{x}))&=\sum_{\sigma \in \Sym_{J_1}\times \cdots \times\Sym_{J_{\ell+s}}} \epsilon(\sigma) \sigma (P(\mathbf{x})\tilde{\Delta}_{I_1}\cdots\tilde{\Delta}_{I_{\ell+s}}\prod_{i \in {I_{\ell+1}}\cup \cdots \cup {I_{\ell+s}}}\mathbf{x}_i) 
\\
&= \sum_{\sigma \in \Sym_{J_1}\times \cdots \times\Sym_{J_{\ell+s}}} \epsilon(\sigma) \sigma \left(Q(\mathbf{x}^2)\tilde{\Delta}_{I_1}\cdots\tilde{\Delta}_{I_{\ell+s}}\prod_{i \in {J_{\ell+1}}\cup \cdots \cup {J_{\ell+s}}}\mathbf{x}_i\right)
\\ & = \prod_{i \in {J_{\ell+1}}\cup \cdots \cup {J_{\ell+s}}}\mathbf{x}_i\sum_{\sigma \in \Sym_{J_1}\times \cdots \times\Sym_{J_\ell}} \epsilon(\sigma) \sigma \left(Q(\mathbf{x}^2)\tilde{\Delta}_{I_1}\cdots\tilde{\Delta}_{I_{\ell+s}}\right)
\\ & = k_1! \cdots k_\ell! \cdot r_{\ell+1}! \cdots r_{\ell+s}!  \tilde{\Delta}_{J_1}\cdots\tilde{\Delta}_{J_\ell} \tilde{\Delta}_{J_{\ell+1}}\cdots\tilde{\Delta}_{J_{\ell+s}}\prod_{i \in {J_{\ell+1}}\cup \cdots \cup {J_{\ell+s}}}\mathbf{x}_i
\end{aligned}
\]
where the last equality follows from \eqref{eq:DeltaSn}.
Therefore, $I$ contains the polynomial
\[
\tilde{\Delta}_{J_1}\cdots\tilde{\Delta}_{J_\ell} \tilde{\Delta}_{J_{\ell+1}}\cdots\tilde{\Delta}_{J_{\ell+s}}\prod_{i \in {J_{\ell+1}}\cup \cdots \cup {J_{\ell+s}}}\mathbf{x}_i 
\]
which is one of the generators of $I_{\Gamma^*(m)}$. 
By symmetry, we get the theorem.
\end{proof}

\begin{corollary}\label{cor:Ideals}
Let $I \subset \K[\mathbf{x}_1,\ldots,\mathbf{x}_n]$ be a $B_n$-invariant ideal, and let $P \in I$.
Assume that $m$ is a monomial in the leading term of $P$. Using the notation of Definition~\ref{def:Gamma(m)}, assume that  $\wt(P_d) + d_1 + d_2 \leqslant n$. Then, we have the inclusion of varieties 
\[
V(I) \subset V(I_{\Gamma^*(m)})
\]
and $V(I) \cap H_{(\vartheta,\omega)} = \emptyset$ for any bipartition $(\vartheta,\omega) \in \BP_n$ bidominated by $\Gamma^*(m)$.
\end{corollary}
\begin{proof}
The inclusion of varieties follows immediately from the set inclusion in Theorem \ref{thm:Ideals}, while the second claim follows from the set partition of Specht varieties in Theorem \ref{thm: specht varieties and orbit types}.
\end{proof}
We illustrate how this can be applied:
\begin{example}
Let $P = \mathbf{x}_2\mathbf{x}_3(\mathbf{x}_1^2-1) \in \K[\mathbf{x}_1,\ldots,\mathbf{x}_4]$. The polynomial $P$ contains a unique monomial $m = \mathbf{x}_1^2\mathbf{x}_2\mathbf{x}_3$ of highest degree. Using the notation of Definition~\ref{def:Gamma(m)} we obtain $\wt(P_4) = 3, l=1, s=2, d_1 =1,d_2 =0$, $\Gamma (m) = ((2),(1,1))$, and $\Gamma^*(m) = ((1,1),(2))$. Let $I$ denote the ideal that is generated by the $B_n$ orbit of $P$. Then, by Corollary \ref{cor:Ideals} it must be $V(I) \subset V_{((1,1),(2))}$. Thus, we have $$V(I) \subset B_4 \cdot \{(0,0,a,a),(0,0,a,b),(a,a,a,a),(0,0,0,a),(0,0,0,0) : a,b \in \K_{>0}\}.$$
We observe that $V(I) = \{(1,1,1,1),(0,0,0,a),(0,0,0,0) : a \in \K\}$. Thus $V(I)$ contains already points of three of the five possible orbit types.
\end{example}
\subsection{Connections with Representation  Theory} We assume that $\K = \C$, or $\K = \R$ if $G$ is a real reflection group.
Let $G$ be a finite group acting linearly on the polynomial ring $\K[\mathbf{x}_1,\ldots,\mathbf{x}_n]$. The polynomials fixed by this action  form a finitely generated subalgebra $\K[\mathbf{x}_1,\ldots,\mathbf{x}_n]^G$.
Moreover, each finite group admits - up to isomorphism - a finite number of irreducible $\K[G]$-modules and the action on $\K[\mathbf{x}_1,\ldots,\mathbf{x}_n]$ can be decomposed into isotypic components, i.e., we have a decomposition of the form
\begin{equation}\label{eq:iso}\K[\mathbf{x}_1,\ldots,\mathbf{x}_n]=\bigoplus_{\chi} \K[\mathbf{x}_1,\ldots,\mathbf{x}_n]_{\chi},
\end{equation}
where $\chi$ runs over the pairwise non-isomorphic representations and each isotypic component $\K[\mathbf{x}_1,\ldots,\mathbf{x}_n]_{\chi}$ contains only pairwise isomorphic $\K[G]$-submodules. Notice, that with this notion the invariant polynomials $\K[\mathbf{x}_1,\ldots,\mathbf{x}_n]^G$ correspond to the trivial representation. Clearly, $\K[\mathbf{x}_1,\ldots,\mathbf{x}_n]$ has the structure of a $\K[\mathbf{x}_1,\ldots,\mathbf{x}_n]^G$-module. Finally, let $J_{+}\subset \K[\mathbf{x}_1,\ldots,\mathbf{x}_n]$ be the ideal generated by invariant polynomials of positive degree. Then the algebra $\K[\mathbf{x}_1,\ldots,\mathbf{x}_n]_G:=\K[\mathbf{x}_1,\ldots,\mathbf{x}_n] / J_+$ is called the \emph{coinvariant algebra}. Whereas these algebras can be defined and studied for all finite groups, it was shown by Chevalley (\cite[Theorem (B)]{chevalley1955invariants}) and Sephard-Todd (\cite{shep}) that for finite reflection groups, these algebras are very regular and moreover that they charaterize  finite reflection groups by what is now known as Chevalley–Shephard–Todd theorem:

\begin{theorem}[\cite{chevalley1955invariants,shep}] \label{thm:chev1}
Let $G$ be a finite group. Then, the following are equivalent
\begin{enumerate}
    \item $G$ is a group generated by reflections.
    \item The algebra of polynomial invariants $\K[\mathbf{x}_1,\ldots,\mathbf{x}_n]^G$ is a (free) polynomial algebra.
    \item The algebra $\K[\mathbf{x}_1,\ldots,\mathbf{x}_n]$ is a free module over $\K[\mathbf{x}_1,\ldots,\mathbf{x}_n]^G$.
    \item $\K[\mathbf{x}_1,\ldots,\mathbf{x}_n]_G$ affords the regular representation of $G$, i.e.,
\[\K[\mathbf{x}_1,\ldots,\mathbf{x}_n]_G\simeq \bigoplus_{\chi} \dim(\chi) \chi,  \]
where $\chi$ runs over the pairwise non-isomorphic representations of $G$.
\end{enumerate}
\end{theorem}

Let $(\lambda,\mu) \in \BP_n$ be a bipartition, denote by $\K[\mathbf{x}_1,\ldots,\mathbf{x}_n]_{(\lambda,\mu)}$ the isotypic component corresponding to $(\lambda,\mu)$, and by  $I_{(\lambda,\mu)}$ the associated Specht ideal. Notice that by Theorem \ref{thm:chev1}, $\K[\mathbf{x}_1,\ldots,\mathbf{x}_n]_{(\lambda,\mu)}$ is also a finitely generated $\K[\mathbf{x}_1,\ldots,\mathbf{x}_n]^{G}$ module, generated by $s_{\lambda,\mu}^2$ many elements, where   $s_{\lambda,\mu}$ denotes the dimension of the corresponding irreducible representation. It follows from (\cite[Theorem 1 (2)]{morita1998higher}) that $s_{\lambda,\mu}$ is in fact equal to the number of standard bitableaux of shape ${\lambda,\mu}$. 
We note the following proposition:

\begin{proposition} \label{prop:B irred reps}
 Let $d$ be minimal with \[ \mathcal{V} := \K[\mathbf{x}_1,\ldots,\mathbf{x}_n]_{(\lambda,\mu)} \cap \K[\mathbf{x}_1,\ldots,\mathbf{x}_n]_{\leq d} \neq \emptyset. \] Then, the multiplicity of an irreducible representation of type $(\lambda,\mu)$ in $\mathcal{V}$ is $1$. This unique irreducible representation is given by the $B_n$-Specht polynomials of shape $(\lambda,\mu)$. Moreover, $\K[\mathbf{x}_1,\ldots,\mathbf{x}_n]_{(\lambda,\mu)}$ is contained in the ideal generated by this unique irreducible representation, i.e., \[\K[\mathbf{x}_1,\ldots,\mathbf{x}_n]_{(\lambda,\mu)}\subset I_{(\lambda,\mu)}. \]
\end{proposition}

\begin{proof}
Since $\K[\mathbf{x}_1,\ldots,\mathbf{x}_n]_{(\lambda,\mu)}$ is a direct sum of irreducible $B_n$-modules isomorphic to the standard $B_n$-Specht module, $W_{(\lambda, \mu)}$, it is enough to show that for every Specht polynomial $Q= \spe_{(T,S)}$ with $\sh (T,S) = (\lambda, \mu)$, and every $B_n$-isomorphism $\phi$, the polynomial
\[
P=\phi(Q)
\]
is divisible by $Q$. 
First, for every $i\in \{1, \ldots, n\}$, if $\epsilon_i$ is the map replacing $\mathbf{x}_i$ with $-\mathbf{x}_i$, since $\phi$ respects the action of $B_n$, we must have 
\[
\epsilon_i P = \begin{cases}
- P & \text{ if } i \in S
\\
P & \text{ if } i \notin S
\end{cases},
\]
which implies that $P$ is of the form
\[
P(\mathbf{x}) =  \tilde{P}(\mathbf{x}^2)\prod_{i\in S}\mathbf{x}_i. 
\]
Then, for every $\tau$ switching two elements in a same column of $T$ or $S$, we must have $\tau \tilde{P} = - \tilde{P}$, so that $P$ is divisible by $Q$.
\end{proof}

\begin{remark}
We remark that the statement about multiplicity $1$ of an irreducible representation in Proposition \ref{prop:B irred reps} does not apply in general. Consider the real reflection group $D_n \subset B_n$ which is generated by all permutations and those maps that switch an even number of signs. Then the $B_n$-irreducible representation of type $(\lambda,\mu)$ and $(\mu,\lambda)$ remain $D_n$-irreducible if $\lambda \neq \mu$, but are $D_n$-isomorphic. By Theorem~\ref{thm:equivalence} we have 
\[ I_{((2),(1,1))} \not \subset I_{((1,1),(2))} \quad \mbox{and} \quad I_{((1,1),(2))} \not \subset I_{((2),(1,1))}.\] Thus, no polynomial in the $B_n$-orbit of $(\mathbf{x}_3^2-\mathbf{x}_4^2)\mathbf{x}_3\mathbf{x}_4$ divides the polynomial $(\mathbf{x}_1^2-\mathbf{x}_2^2)\mathbf{x}_3\mathbf{x}_4$ although \[ \langle \spe_{(T,S)} : \sh (T,S) = ((2),(1,1)) \rangle_\K \simeq_{D_n} \langle \spe_{(T,S)} : \sh (T,S) = ((1,1),(2)) \rangle_\K. \] The $D_n$-irreducible representation $((2),(1,1))$ occurs for the first time in $\K[\mathbf{x}_1,\ldots,\mathbf{x}_4]_{\leq 4}$ but with multiplicity $2$.
\end{remark}

Combining Corollary \ref{cor:Ideals} with Proposition \ref{prop:B irred reps} we obtain the following:

\begin{theorem}
Let $I\subset\K[\mathbf{x}_1,\ldots,\mathbf{x}_n]$ be a $B_n$-invariant ideal satisfying the conditions of Corollary~\ref{cor:Ideals} and let $(\lambda,\mu)=\Gamma^*(m)$. Consider the the associated  coordinate ring
$$R_I=\K[\mathbf{x}_1,\ldots,\mathbf{x}_n]/I.$$
Then, viewed as a $\K[B_n]$-module, $R_I$ does not contain any irreducible $\K[B_n]$-submodule which is isomophic
to $W_{\lambda,\mu}$. 
Moreover, consider $$\tilde{I}=I\cap \K[\mathbf{x}_1,\ldots,\mathbf{x}_n]^{B_n}, \text{ and } R_{\tilde{I}}=\K[\mathbf{x}_1,\ldots,\mathbf{x}_n]^{B_n}/\tilde{I}$$  the corresponding quotient. Then, $R_I$ is a finite $R_{\tilde{I}}$ module  of rank bounded by $\sum_{(\vartheta,\omega)\not\unlhd(\lambda,\mu)} s_{\vartheta,\omega}^2$.
\end{theorem}
\begin{proof}
 By Theorem \ref{thm:chev1} (3) we have that $\K[\mathbf{x}_1,\ldots,\mathbf{x}_n]$ is a free $\K[\mathbf{x}_1,\ldots,\mathbf{x}_n]^{B_n}$ module, i.e., it is possible to find a basis   $\mathcal{B}=\{b_1,\ldots,b_{m}\}$, with $m=|B_n|=2^{n}\cdot n!$, such that every $f\in\K[\mathbf{x}_1,\ldots,\mathbf{x}_n]$ can be uniquely represented with respect to $\mathcal{B}$ by  coefficients in $\K[\mathbf{x}_1,\ldots,\mathbf{x}_n]^{B_n}$. Now it follows directly that for any basis $\mathcal{B}$ we obtain a generating  set for $R_{I}$ over $R_{\tilde{I}}$ by  the elements $\phi_{I}(b_i)\in R_{I}$, where $\phi_I$ denotes the canonical homomorphism. Thus,  $\R_I$ is a finite module over $\R_{\tilde{I}}$ - not necessarily free. To establish the announced statements, we first remark that   we can assume that $\mathcal{B}$ respects the decomposition in \eqref{eq:iso}, i.e., that $\mathcal{B}$ can be partitioned in  such a way that for every $(\vartheta,\omega) \in \BP_n$ there exists a unique subset $\mathcal{B}_{(\vartheta,\omega)}\subset\mathcal{B}$ consisting of  $s_{\vartheta,\omega}^2$ many elements which generate the free module $\K[\mathbf{x}_1,\ldots,\mathbf{x}_n]_{(\vartheta,\omega)}$. Indeed, for example  the so called \emph{higher Specht polynomials} explicitly constructed in \cite{morita1998higher} constitute such a basis. Since Proposition  ~\ref{prop:B irred reps} and Corollary~\ref{cor:Ideals} yield that  $\K[\mathbf{x}_1,\ldots,\mathbf{x}_n]_{(\vartheta,\omega)}\subset I_{(\vartheta,\omega)}\subset I$ for all bipartitions bidominated by $(\lambda,\mu)$ we immediately see that $\phi_I(\mathcal{B}_{(\vartheta,\omega)})=\{0\}$ for all such bipartitions. 
Therefore, already the set \[\phi_{I}\left(\bigcup_{(\vartheta,\omega) \not \unlhd (\lambda,\mu)}\mathcal{B}_{(\vartheta,\omega)}\right)\] 
generates $R_{I}$, which directly yields the announced bound on the rank. Furthermore, since the action on $R/\tilde{I}$ is trivial we can conclude that the only irreducible representations which appear in $R/I$ are the ones  stemming from this basis and the associated  list of partitions, i.e., $R_I$ does not contain any irreducible $\K[B_n]$-submodule which is isomophic
to $W_{\vartheta,\omega}$, for bipartitions $(\vartheta,\omega)$ which are bidominated by $(\lambda,\mu)$. This completes the proof.  
\end{proof}

\section{Conclusion and open questions} \label{sec:ciao}

We initiated  in this article the investigation of a class of polynomial ideals which are naturally linked to the action of a group on a polynomial ring. Our results provide an analogue of the relation of the combinatorics of integer partitions and $\mathcal{S}_n$-Specht ideals to bipartitions and $B_n$-Specht ideals. The present work shows that it is indeed possible to derive an analogous connection between combinatorics and algebra for the case of the hyperoctahedral group as was observed in the case of the symmetric group. Both groups are finite reflection groups, and they thus share important similarities  from a view point of invariant theory and representation theory. Our results here lead to the natural question, if similar relations between integer (bi)-partitions and ideals can be derived for other (pseudo)-reflection groups. Indeed, in \cite{morita1998higher} a similar basis of the coinvariant algebra is provided for complex reflection groups of type $G(r,p,n)$, where $r,p,n \in \Z_{\geq 1}$ and $p \mid n$. We recover $G(1,1,n) \simeq \mathcal{S}_n$, $G(2,1,n) \simeq B_n$, and $G(2,2,n) \simeq D_n$. It seems plausible to envision similar results to the ones presented here in these cases as well. More precisely, that there is a partial order on $r$-multipartitions, which are linked  to  the irreducible representations of the complex reflection group $G(r,1,n)$, which transfers to  the inclusion of the $G(r,1,n)$-Specht ideals and their corresponding varieties. 
Furthermore, it remains to investigate if the $B_n$-Specht ideals also have similarly nice algebraic properties as their  $S_n$ counter parts. Indeed, it is known that the $\mathcal{S}_n$-Specht ideals are radical (see \cite[Theorem 1.1]{murai2021note} and \cite[Proposition 4]{woo2005ideals}). Both proofs rely on the understanding of the $\mathcal{S}_n$-Specht varieties in terms of orbit sets, i.e., Theorem \ref{thm: specht varieties and orbit types}, and crucially depend on the property that any $\mathcal{S}_n$-orbit set is non-empty which is not true for $B_n$-Specht varieties. Nevertheless, computational evidence for small number of variables motivates the conjecture, similar to the $S_n$ situation (\cite{lien2021symmetric,woo2005ideals,murai2021note,OW}).
\begin{conjecture}
The $B_n$-Specht ideals are radical. Moreover, for a bipartition $(\lambda,\mu) \in \BP_n$ the $B_n$-Specht polynomials $\{\spe_{(T,S)} : (T,S) \mbox{ is a bitableau of shape } (\vartheta,\omega) \unlhd (\lambda,\mu)\}$ form a universal Gröbner basis of $I_{(\lambda,\mu)}$.
\end{conjecture}
Finally, Yanagawa \cite{yanagawa2021specht} classified the partitions for which the associated $S_n$-Specht ideals are Cohen-Macaulay and it would be interesting to derive a similar characterization of bipartitions $(\lambda,\mu)$ for which the corresponding  $B_n$-Specht ideals have this property. 
\section*{Acknowledgements}
The authors would like to thank two anonymous reviewers for their careful reading of the manuscript, insightful suggestions and comments. 
\printbibliography
\end{document}